\documentclass[11pt, oneside]{article}   	% use "amsart" instead of "article" for AMSLaTeX format
\usepackage{geometry}                		% See geometry.pdf to learn the layout options. There are lots.
\usepackage{graphicx}				% Use pdf, png, jpg, or eps§ with pdflatex; use eps in DVI mode
\usepackage{amssymb,amsmath,amsfonts,bbm}
\usepackage{amsthm, thmtools,theoremref}
\usepackage[utf8]{inputenc}
\usepackage{xcolor}
\usepackage{parskip}
\usepackage{bbm}
\usepackage{verbatim}
\usepackage{hyperref}
\usepackage{csquotes}
\usepackage{appendix}

\usepackage{natbib}
\DeclareMathOperator*{\esssup}{ess\,sup}

\usepackage{fancyhdr}
\allowdisplaybreaks

\theoremstyle{plain}
\newtheorem{thm}{Theorem}%[section]
\newtheorem{lem}[thm]{Lemma}
\newtheorem{prop}[thm]{Proposition}
\newtheorem{cor}[thm]{Corollary}

\theoremstyle{definition}
\newtheorem{defn}[thm]{Definition}
\newtheorem{exmp}[thm]{Example}
\newtheorem{assume}[thm]{Assumption}

\theoremstyle{remark}
\newtheorem{rem}[thm]{Remark}

\title{Utility maximization with ratchet and drawdown constraints on consumption in incomplete semimartingale markets}
\author{Anastasiya Tanana\thanks{Department of Mathematics, The University of Texas at Austin, \textit{atanana@utexas.edu}.\newline \textit{MSC2020 subject classifications:} Primary 93E20; secondary 91G80, 91B08. \newline \textit{Keywords and phrases:} utility maximization, convex duality, ratcheting of consumption, drawdown constraint on consumption, running maximum, incomplete markets.}}
\date{July 13, 2022}							% Activate to display a given date or no date

\begin{document}
\maketitle

\begin{abstract}{\footnotesize
In this paper, we study expected utility maximization under ratchet and drawdown constraints on consumption in a general incomplete semimartingale market using duality methods. The optimization is considered with respect to two parameters: the initial wealth and the essential lower bound on consumption process. In order to state the problem and define the primal domains, we introduce a natural extension of the notion of running maximum to arbitrary non-negative optional processes and study its properties. The dual domains for optimization are characterized in terms of solidity with respect to an ordering that is introduced on the set of non-negative optional processes. The abstract duality result we obtain for the optimization problem is used in order to derive a more detailed characterization of solutions in the complete market case.}
\end{abstract}

%{\footnotesize \tableofcontents}

\section{Introduction}\label{sec:intro}

The first mathematical study of optimal investment and consumption of an agent with intolerance for any decline in the standard of living appears in \cite{dybvig}. The constraint that the consumption process is non-decreasing, also called \textit{ratcheting of consumption}, can be seen as an extreme form of habit formation.

{\footnotesize
\begin{displayquote}
``The model in this paper is close to modern models of habit formation such as those of Constantinides (1990), Detemple and Zapareto (1991), Ingersoll (1992), Shrikhande (1992), or Sunderesan (1989). The main difference is in the rapidity (immediate) of the habit formation and the severity (lexicographic) of the agent's preferences for maintaining a new standard of living." \hfill \citet[p. 289]{dybvig}
\end{displayquote}}
Dybvig suggests that this model might be suitable for situations where consumption involves long-term commitments.

{\footnotesize
\begin{displayquote}
``For example, it may be a good model of a university or foundation that at least some of the expenditures cannot be decreased quickly, due to implicit and explicit long-term commitments to faculty and due to the commitments to donors about the use of buildings or equipment."

\hfill\citet[p. 288]{dybvig}
\end{displayquote}}
%%%%%%%%%%%%
By several changes of variables, a constraint that consumption can fall no faster than at some given rate (i.e., the constraint of the type $c_t\geq e^{-\int_s^t\alpha_udu}c_s$ for $t\geq s\geq 0$, where $\alpha_\cdot$ is given) can be transformed into ratcheting constraint (see Section 2 in \cite{dybvig} for the case when $\alpha_\cdot$ is constant), thus generalizing Dybvig's model to a less rigid form of habit formation.

\cite{dybvig} finds the optimal investment and consumption strategies for an infinitely-lived ratchet investor with CRRA utility function in a market with underlying risky asset following the Geometric Brownian Motion (GBM) by considering the corresponding Hamilton--Jacobi--Bellman (HJB) equation. \cite{riedel} shows that, in a complete market with pricing kernel driven by a L\'evy process, an optimal consumption plan of an infinitely-lived ratchet investor with an arbitrary utility function is in fact equal to the running maximum of the optimal consumption plan of an unconstrained investor with the same utility function. Riedel's proof is elementary in the sense that, essentially, it relies on a concavity argument and integrations by parts. \cite{Koo:2012aa} generalize Dybvig's portfolio selection result under the same assumptions on the market to an arbitrary utility function by using duality and the Feynman-Kac formula. \cite{Watson-Scott} extend Riedel's result to finite time horizon by introducing a deterministic function of time, called a coupling curve, that reflects the effects of finiteness of horizon in optimization and is described as a free boundary of a certain free-boundary problem. \cite{Jeon-Koo-Shin2018} study the finite horizon GMB market case with general utility function by transforming it into an infinite family of optimal stopping problems.

The working paper of \cite{Arun2012} is the first to study optimal investment and consumption problem under a \textit{drawdown constraint on consumption}. Under this condition, the consumption is not allowed to fall below a fixed proportion $\lambda\in[0,1]$ of the running maximum of past consumption. In particular, $\lambda=0$ corresponds to the unconstrained problem and $\lambda=1$ corresponds to consumption ratcheting. \cite{Arun2012} finds the optimal portfolio and consumption in a GBM market over infinite time horizon for an agent with CRRA utility function by writing down the HJB equation and using duality and a modification of Dybvig's verification argument. \cite{Jeon:2021aa} extend Arun's result to a general class of utility functions by deriving a dual problem consisting of only the choice of an optimal (non-decreasing) maximum consumption process, converting it into an infinite two-dimensional family of optimal stopping problems, and characterizing the solutions to the latter by a family of free boundaries depending on the state variable of the maximum process.

The current paper studies both the ratchet constraint and the drawdown constraint on consumption in general incomplete semimartingale markets by duality methods. Clearly, the ratchet constraint is a special case of the drawdown constraint, but in this paper the case $\lambda=1$ is thought of as the basic one and the solution for $\lambda\in(0,1)$ is derived by, loosely speaking, an interpolation between the ratchet investor problem with $\lambda=1$ and the unconstrained problem with $\lambda=0$.

Our market model is taken from \cite{mostovyi} with an additional assumption that the stochastic clock is equivalent to the Lebesgue measure on $[0,\hat{T})$, where $\hat{T}$ is either finite or infinite time horizon. \cite{mostovyi} provides a simple necessary and sufficient condition, namely, finiteness of the primal and dual value functions, for the key assertions of utility maximization theory to hold in an incomplete semimartingale market model with intermediate consumption, stochastic clock, and utility stochastic field. This approach parallels the work of \cite{Kramkov:2003aa} on maximization of the expected value of deterministic utility from terminal wealth. A different but common approach to establish duality for expected utility maximization with time- and/or scenario-dependent utility is to require the utility to satisfy additional assumptions, including uniform reasonable asymptotic elasticity of \cite{Karatzas-Zitkovic2003} (see also \cite{Zitkovic2005}, \cite{BK}, \cite{yu}), a generalization of the reasonable asymptotic elasticity condition of \cite{Kramkov:1999aa} for deterministic utility applied to terminal wealth.

The main difficulty in applying Mostovyi's result for the ratchet/drawdown constraint is identifying the appropriate primal and dual domains. We handle this issue by introducing a natural extension of the notion of the running maximum to arbitrary non-negative optional processes, which we call a \textit{running essential supremum}, and defining the primal domain as the solid hull of all consumption plans satisfying the ratchet/drawdown constraint (formulated in terms running essential supremum) together with the budget constraint. The dual domains, defined as the polar sets of the primal domains in the sense of \cite{bipolar}, are characterized via a family of orderings we introduce on the set of non-negative optional processes. The corresponding ordering for the ratchet constraint, which we call \textit{chronological ordering}, implicitly appears in convex duality method for optimization over the set of increasing processes described in \cite{BK}, however, the approach of the current paper (i) seems to be more direct, (ii) allows for generalization to the drawdown constraint with $\lambda\in(0,1)$, (iii) allows to add another parameter to optimization, an essential lower bound on the consumption process. This parameter is omnipresent in the literature on optimization under the ratchet/drawdown constraint, since, if the market is Markovian, adding this second parameter turns the problem into Markovian: all the information about the past consumption that is necessary for the future optimization is contained in the current running maximum which serves (up to multiplication by $\lambda$) as a lower bound for the future consumption. Based on \cite{mostovyi}, we derive a duality result for the \textit{two-parameter} optimization problem, where the parameters are the initial wealth and the essential lower bound on consumption. In this regard, the formulation and derivation of our main optimization result (Theorem \ref{thm:main-duality}) is similar to the derivation of \cite{hug-kramkov} from the results of \cite{Kramkov:1999aa,Kramkov:2003aa}, and to the work of \cite{yu}.

In a complete market, where the set of equivalent martingale deflators is a singleton, the characterization of dual domains simplifies (Proposition \ref{prop:complete-case-D}) and allows for a more detailed description of the structure of optimizers. It turns out that in the drawdown constraint case the optimizers exhibit tree possible types of behavior: the agent consumes either at the minimal level currently allowed by the drawdown constraint, or at the current running essential supremum level, or as an unconstrained agent (i.e., not restricted by the drawdown constraint) with the same utility function but a different initial wealth. In the case of ratchet constraint, the description of optimizers is linked to the Bank--El Karoui Representation Theorem for stochastic processes and to the related notion of envelope process introduced in \cite{BK}. As a special case, we derive from Corollary~\ref{cor:complete-env} Riedel's formula for the optimal consumption plans in a complete L\'evy market model with exponential time preferences and infinite time horizon.

The paper is organized as follows. In Section \ref{sec:math-framework}, we describe the market model, define the domains for consumption processes, and formulate the utility maximization problem. In Section \ref{sec:run-esssup}, we define the running essential supremum process and study its properties. Section \ref{sec:domains} describes the structure of the Brannath--Schachermayer polar sets of the consumption domains. Section \ref{sec:optimization} introduces the two-parameter families of primal and dual domains and proves the main optimization result, Theorem \ref{thm:main-duality}. In Section \ref{sec:complete}, we derive more specific results for the complete market case. In the \hyperref[app:envelope]{Appendix}, we give a modification of a lemma by \cite{BK} concerning existence and uniqueness of the envelope process allowing us to say more about the structure of optimizers for the ratchet constraint in a complete market and, in fact, give an alternative solution (Proposition \ref{prop:alternative-sol}), not relying on the main duality result, for this case.

%%%%%%%%%%%%%%%%%
\section{Mathematical framework}\label{sec:math-framework}

%%%%%%%%%%%%%%%%%

\subsection{The market model}

Let $\hat{T}\in(0,\infty]$ be a deterministic time horizon and denote $\mathcal{I}:=[0,\hat{T}]$ if $\hat{T}<\infty$ and $\mathcal{I}:=[0,\infty)$ if $\hat{T}=\infty$. We consider a market consisting of one num\'eraire asset (for example, a savings account) and $d$ risky assets with discounted price process $S=(S)_{t\in\mathcal{I}}$ given by an $\mathbb{R}^d$-valued semimartingale on a filtered probability space $(\Omega,\mathcal{F},(\mathcal{F}_t)_{t\in\mathcal{I}},\mathbb{P})$, where the filtration $(\mathcal{F}_t)_{t\in\mathcal{I}}$ satisfies the usual conditions. We denote the corresponding optional $\sigma$-algebra on $\Omega\times[0,\hat{T})$ by $\mathcal{O}$.

We fix a stochastic clock $\kappa=(\kappa_t)_{t\in[0,\hat{T}]}$ representing the notion of time according to which consumption is assumed to occur.
\begin{assume}\label{ass:clock}
There is a \textit{strictly positive}, adapted process $\dot{\kappa}_s:\mathcal{I}\to(0,\infty)$ such that
$\kappa_t(\omega)=\int_0^t \dot{\kappa}_s(\omega)ds$ for all $(\omega,t)\in\Omega\times[0,\hat{T}]$.
Additionally, $\kappa_t$ is uniformly bounded: $\kappa_{\hat{T}}(\omega)\leq A$ for a finite constant $A$ and for all $\omega\in\Omega$.
\end{assume}
Two examples of stochastic clocks satisfying these conditions are the exponential clock $\kappa(t)=\frac{1}{\nu}(1-e^{-\nu t})$, $t\geq 0$, for some $\nu>0$ and for any $\hat{T}\in(0,\infty]$, and the Lebesgue clock $\kappa(t)=t$ for $\hat{T}<\infty$. Assumption~\ref{ass:clock} implies, in particular, that the measures $\mathbb{P}\times d\kappa$ and $\mathbb{P}\times \text{Leb}([0,\hat{T}))$ on $\Omega\times[0,\hat{T})$ are equivalent.

A portfolio is defined as a triplet $\Pi=(x,H,c)$ of a constant initial wealth $x$, a predictable $\mathbb{R}^d$-valued $S$-integrable process $H=(H_t)_{t\in\mathcal{I}}$ of risky asset quantities, and a non-negative finite-valued optional process $c:\Omega\times[0,\hat{T})\to[0,\infty)$ representing the consumption rate relative to $\kappa$.
The discounted value process $V=(V_t)_{t\in\mathcal{I}}$ of portfolio $\Pi$ is
\begin{equation}\label{eq:wealth-process}
V_t:=x+\int_0^t H_sdS_s-\int_0^tc_sd\kappa_s,\quad t\in\mathcal{I}.
\end{equation}
For $x>0$, a consumption plan $c$ is called \textit{$x$-admissible} if there exists a predictable $\mathbb{R}^d$-valued $S$-integrable process $H$ such that the value process $V$ of portfolio $(x,H,c)$ is non-negative.

\begin{rem}\label{rem:on-model}
\begin{enumerate}
\item[(i)] Our market model is essentially the model of \cite{mostovyi} but with additional assumptions on the stochastic clock $\kappa$: it is absolutely continuous with respect to the Lebesgue measure and has support of the form $\mathcal{I}=[0,\hat{T}]\cap[0,\infty)$. Clearly, if $\hat{T}<\infty$ then we can embed our model into the infinite horizon Mostovyi's model by assuming that $S_t=S_{\hat{T}}$, $\kappa_t=\kappa_{\hat{T}}$, $\mathcal{F}_t=\mathcal{F}_{\hat{T}}$ for all $t\in(\hat{T},\infty)$. This allows us to freely apply Mostovyi's duality results in what follows. The reason why we separate the finite and infinite time horizon cases is because an $x$-admissible consumption plan $c$ satisfying the drawdown constraint on $[0,\hat{T}]$ for $\hat{T}<\infty$ generally cannot be extended to an $x$-admissible consumption plan on $[0,\infty)$ still satisfying the drawdown constraint on $[0,\infty)$: if we extend $c$ by taking $c_t=0$ for $t\geq \hat{T}$, the drawdown constraint fails after time $\hat{T}$. Hence, the drawdown constraints on consumption throughout $[0,\infty)$ and on $[0,\hat{T}]$ for $\hat{T}<\infty$ are two different (though similar) conditions, and we have to separate their treatment. The results of this paper remain valid if we assume $\hat{T}$ to be a $[0,\infty]$-valued stopping time with respect to a filtration $(\mathcal{F}_{t})_{t\in[0,\infty)}$ satisfying the usual conditions (with the exception of supplementary results in the \hyperref[app:envelope]{Appendix}, which rely on the Representation Theorem of \cite{bank-el-karoui} requiring the stopping time $\hat{T}$ to be predictable, according to their Remark 2.1). In this case, $\mathcal{I}$ becomes a stochastic interval and we assume all processes after time $\hat{T}$ to be equal to their value at time $\hat{T}$ when $\hat{T}$ is finite.
\item[(ii)] Since we work in discounted units, $V_t$ in \eqref{eq:wealth-process} is given in the number of num\'eraire assets, and so is the cumulative consumption process $C_t:=\int_0^t c_s d\kappa_s$, $t\in\mathcal{I}$. The initial model of \cite{mostovyi} allows the flexibility of changing $\dot\kappa_t$ to $\dot\kappa_t/n_t$ and $c_t$ to $c_tn_t$, where $n_t$ is a strictly positive optional process, so that $C_t$ remains the same and the uniform boundedness in Assumption \ref{ass:clock} is satisfied for the new stochastic clock with density $\dot\kappa_tn_t$. In the utility maximization problem, \eqref{eq:primal-problem} below, optimizing over $c_t$ can be reformulated as optimizing over $c_tn_t$ by a simple change of utility function $U$, as described in \cite[Remark 2.2]{mostovyi}. However, in our case, this flexibility is no more useful since $c_t$ is the consumption rate on which we want to impose the ratchet/drawdown constraint and hence is uniquely determined: multiplication by time-dependent process $n_t$ does not preserve this type of constraint. Thus, if we are imposing a ratchet/drawdown constraint on consumption rate measured in currency per unit time (with respect to the usual Lebesgue clock $dt$), $\dot\kappa_t$ is uniquely determined from \eqref{eq:wealth-process} through the num\'eraire asset: if the num\'eraire asset is $N_t:=N_0\exp\left(\int_0^t r_sds\right)$, where  $N_0>0$ and an optional process $r_t$ represents an interest rate, then the appropriate stochastic clock density is given by $\dot\kappa_t=1/N_t$ and Assumption \ref{ass:clock} must be satisfied for this clock. On the other hand, if we are imposing a ratchet/drawdown constraint on discounted consumption rate (i.e., measured in number of num\'eraire assets per unit time with respect to $dt$) then the appropriate stochastic clock density is $\dot\kappa_t\equiv1$. In particular, this stochastic clock satisfies Assumption \ref{ass:clock} only if $\hat{T}$ is finite (or uniformly bounded in case of stopping time $\hat{T}$). If we want to impose a ratchet/drawdown constraint on consumption rate measured in a different asset, this can be done in a similar way as long as Assumption \ref{ass:clock} for the corresponding clock $\kappa$ is satisfied.
\end{enumerate}
\end{rem}

Next, we define the set of all non-negative value processes associated with portfolios of the form $\Pi=(1,H,0)$,
$$\mathcal{X}:=\left\{X\geq 0:\ X_t=1+\int_0^t H_sdS_s \text{ for } t\in\mathcal{I}\right\},$$
and the set of \textit{equivalent martingale deflators},
\begin{align*}
\mathcal{Z}:=\left\{\right.&\left.(Z_t)_{t\in\mathcal{I}}>0:\ Z \text{ is a c\`adl\`ag martingale such that } Z_0=1 \text{ and}\right.\\
&\left. XZ=(X_tZ_t)_{t\in\mathcal{I}} \text{ is a local martingale for every } X\in\mathcal{X}\right\}.
\end{align*}
We make the following assumption, which is related to the absence of arbitrage in the market.
\begin{assume}\label{ass:NA}$\mathcal{Z}\neq\emptyset$.\end{assume}
If $\hat{T}$ is finite and $S$ is locally bounded then $S$ is a local martingale under any probability measure $\mathbb{Q}\sim \mathbb{P}$ given by $\frac{d\mathbb{Q}}{d\mathbb{P}}=Z_{\hat{T}}$ with $Z\in\mathcal{Z}$. And vice versa, if $S$ is a local martingale under a probability measure $\mathbb{Q}\sim \mathbb{P}$ then $Z_t=\frac{d\mathbb{Q}}{d\mathbb{P}}\big\vert_{\mathcal{F}_{t}}$, $t\in\mathcal{I}$, belongs to $\mathcal{Z}$. Hence in this case, by \cite[Corollary 1.2]{delbaen_general_1994}, Assumption \ref{ass:NA} is equivalent to no free lunch with vanishing risk (NFLVR) condition on the market, a version of no arbitrage condition.

According to \cite{mostovyi}, Lemma 4.2, a consumption plan $c$ is $x$-admissible if and only if
\begin{equation}\label{eq:x-admissible}
\mathbb{E}\left[\int_0^{\hat{T}} c_tZ_td\kappa_t\right]\leq x,\quad \forall Z\in\mathcal{Z}.
\end{equation}
We adopt the following simplifying notation for non-negative optional processes $c$ and $\delta$,
 $$\langle c,\delta\rangle:=\mathbb{E}\left[\int_0^{\hat{T}} c_t\delta_t d\kappa_t\right],$$
 so, in particular, the $x$-admissibility condition \eqref{eq:x-admissible} can be written as $\sup_{Z\in\mathcal{Z}}\langle c,Z\rangle\leq x$.
 \begin{rem}
 The results of this paper hold as well if the set $\mathcal{Z}$ of equivalent martingale deflators is replaced by a set
 \begin{align*}
\mathcal{Z}':=\left\{\right.&\left.(Z_t)_{t\in\mathcal{I}}>0:\ Z \text{ is a c\`adl\`ag local martingale such that } Z_0=1 \text{ and}\right.\\
&\left. XZ=(X_tZ_t)_{t\in\mathcal{I}} \text{ is a local martingale for every } X\in\mathcal{X}\right\}
\end{align*}
of \textit{equivalent local martingale deflators}. Clearly, $\mathcal{Z}\subseteq\mathcal{Z}'$ and the assumption $\mathcal{Z}'\neq\emptyset$ is weaker than Assumption \ref{ass:NA}. By Proposition~1 in \cite{mostovyiNUPBR}, assumption $\mathcal{Z}'\neq\emptyset$ is equivalent to another version of no arbitrage condition -- no unbounded profit with bounded risk (NUPBR). Under this assumption, a consumption plan $c$ is $x$-admissible if and only if $\sup_{Z\in\mathcal{Z}'}\langle c,Z\rangle\leq x$, analogously to \eqref{eq:x-admissible} (see Lemma~1 in \cite{mostovyiNUPBR}), which makes it possible to replace $\mathcal{Z}$ with $\mathcal{Z}'$ in the arguments that follow.
 \end{rem}

%%%%%%%%%%%%%%%%%%%%%%%%%%%%%

\subsection{Domains for consumption processes}\label{subsec:domains}

In order to be able to formulate the drawdown constraint on consumption rate in this generalized framework, it is necessary to have a meaningful extension of the notion of running maximum from continuous processes to processes that are just measurable. Such an extension, which we call a running essential supremum process, is introduced in Definition \ref{def:running-max} and studied in Section \ref{sec:run-esssup}. The running essential supremum $\bar{c}$ of a non-negative optional process $c$ belongs (by Proposition \ref{prop:run-sup-properties}) to the following class of stochastic processes:
\begin{defn}\label{def:C-0-class}
The class $\mathcal{C}_{\text{inc}}$ consists of predictable processes $c:\Omega\times[0,\hat{T})\to[0,\infty]$ with non-decreasing, left-continuous paths starting from $c_0=0$.
\end{defn}
Moreover, $\bar{c}$ is the smallest element of $\mathcal{C}_{\text{inc}}$ that is greater or equal than $c$, $\mathbb{P}\times d\kappa-$almost surely on $\Omega\times[0,\hat{T})$ (Propositions \ref{prop:increasing-dom}, \ref{prop:minimality-of-c-bar}).

Now we can formulate the constraint and define the corresponding domains. Let $q\in\mathbb{R}$ and $\lambda\in[0,1]$. We introduce the following drawdown constraint on consumption:
\begin{equation}\label{cond:ddc-with-q}
c_t(\omega)\geq \lambda\cdot(\bar{c}_t(\omega)\vee q),\quad \mathbb{P}\times d\kappa-\text{a.e. on }\Omega\times[0,\hat{T}),\tag{DC$_{\lambda,q}$}
\end{equation}
where ``$\vee$'' denotes the pointwise maximum of two processes (one of which is the constant process $q$ in this case). Similarly, we reserve the notation ``$\wedge$'' for pointwise minimum of two processes. Clearly, when $q\leq 0$, the constraint \label{def:ddc-with-q} reduces to simply
\begin{equation}\label{cond:ddc}
c_t(\omega)\geq \lambda\bar{c}_t(\omega),\quad \mathbb{P}\times d\kappa-\text{a.e. on }\Omega\times[0,\hat{T}).\tag{DC$_{\lambda}$}
\end{equation}
For $x>0$, $\lambda\in[0,1]$, and $q\in\mathbb{R}$, we define the primal domains as follows:
\begin{equation}\label{def:primal-dom}
\mathcal{C}^\lambda(x,q):=\{c\geq 0\text{ optional}:\ c\vee \lambda(\bar{c}\vee q) \text{ is } x\text{-admissible} \},
\end{equation}
and denote $\mathcal{C}^\lambda(x):=\mathcal{C}^\lambda(x,0)$, the domain for optimization without parameter $q$, and $\mathcal{C}^\lambda:=\mathcal{C}^\lambda(1)$ so that $\mathcal{C}^\lambda(x)=x\cdot\mathcal{C}^\lambda$. The set $\mathcal{C}^\lambda(x,q)$ is non-empty as long as the constant process $c\equiv\lambda(0\vee q)$ is $x$-admissible. By \eqref{eq:x-admissible}, this holds if and only if $\lambda q\cdot\sup_{Z\in\mathcal{Z}}\mathbb{E}\left[\int_0^{\hat{T}}Zd\kappa\right]\leq x$.

The sets $\mathcal{C}^\lambda(x,q)$ are solid: if $c'\geq c\geq 0$, $\mathbb{P}\times d\kappa-$a.e., and $c'\in\mathcal{C}^\lambda(x,q)$ then $c\in\mathcal{C}^\lambda(x,q)$. As the following proposition states, $\mathcal{C}^\lambda(x,q)$ is in fact the solid hull of all $x$-admissible processes satisfying \eqref{cond:ddc-with-q}. We take definition \eqref{def:primal-dom} as our working definition of the domains, rather than the solid hull definition \eqref{def:primal-dom2}, because it provides a concrete way of checking whether a \textit{given} process $c$ belongs to $\mathcal{C}^\lambda(x,q)$: by checking whether $c\vee\lambda(\bar{c}\vee q)$ is $x$-admissible.
\begin{prop}\label{prop:C-is-closed}
Let $x>0$ and $q\in\mathbb{R}$. Every $x$-admissible process satisfying \eqref{cond:ddc-with-q} belongs to $\mathcal{C}^\lambda(x,q)$. If $c\in\mathcal{C}^\lambda(x,q)$ then $c\vee \lambda(\bar{c}\vee q)$ satisfies \eqref{cond:ddc-with-q} and belongs to $\mathcal{C}^\lambda(x,q)$.  As a consequence,
\begin{equation}\label{def:primal-dom2}
\begin{aligned}
\mathcal{C}^\lambda(x,q)=\left\{\right.&\left. c\geq 0\text{ optional}:\ \exists\ c'\geq 0 \text{ s.t. } c\leq c',\ \mathbb{P}\times d\kappa-\text{a.e.},\right.\\
&\left. c' \text{ satisfies }\eqref{cond:ddc-with-q},\text{ and }c'\text{ is } x\text{-admissible}\right\}.
\end{aligned}
\end{equation}
\end{prop}
\begin{proof}
For  every $x$-admissible process $c$ satisfying \eqref{cond:ddc-with-q}, $c\vee \lambda(\bar{c}\vee q)=c$, $\mathbb{P}\times d\kappa-$a.e., hence $c\vee \lambda(\bar{c}\vee q)$ is $x$-admissible and $c$ belongs to $\mathcal{C}^\lambda(x,q)$ by definition. It is easy to check directly from Definition \ref{def:running-max} that for every non-negative optional $c$ the running essential supremum of the process $c\vee \lambda(\bar{c}\vee q)$ is equal to $\bar{c}\vee \lambda q\mathbbm{1}_{(0,\hat{T})}$ and, therefore, $c\vee \lambda(\bar{c}\vee q)$ automatically satisfies \eqref{cond:ddc-with-q}. Therefore, if $c\in\mathcal{C}^\lambda(x,q)$ then $c\vee \lambda(\bar{c}\vee q)$ is $x$-admissible and satisfies \eqref{cond:ddc-with-q}, hence $c\vee \lambda(\bar{c}\vee q)$ belongs to $\mathcal{C}^\lambda(x,q)$ by the first assertion. To show \eqref{def:primal-dom2}, we take $c':=c\vee \lambda(\bar{c}\vee q)$ and use the first two assertions.
\end{proof}

Proposition \ref{prop:C-is-closed} implies that if a functional $\mathbb{U}$ on $\mathcal{C}^\lambda(x,q)$ satisfies the monotonicity property
\begin{equation}\label{eq:monotone-functional}
c_1\leq c_2,\quad \mathbb{P}\times d\kappa-\text{a.e.}\quad \Rightarrow\quad \mathbb{U}(c_1)\leq\mathbb{U}(c_2),
\end{equation}
and has a maximizer $c$ on $\mathcal{C}^\lambda(x,q)$, then $\hat{c}:=c\vee \lambda(\bar{c}\vee q)\geq c$ is also a maximizer and, in particular, \textit{it is a maximizer over all $x$-admissible processes satisfying \eqref{cond:ddc-with-q}}.

If $\lambda=0$ then the domains $\mathcal{C}^\lambda(x,q)=\mathcal{C}^0(x)$ consist of all $x$-admissible consumption plans and this case is handled in \cite{mostovyi}, even with more general assumptions on the stochastic clock $\kappa$. Therefore, we will only consider the optimization problem for $\lambda\in(0,1]$ in this paper. If $\lambda=1$ then we can say slightly more: the constraint \eqref{cond:ddc-with-q} turns into $c_t\geq \bar{c}_t\vee q$, $\mathbb{P}\times d\kappa-$a.e., which, by Proposition \ref{prop:increasing-dom}, holds if and only if $c=\bar{c}\geq q$, $\mathbb{P}\times d\kappa-$a.e. The definition \eqref{def:primal-dom} for $\lambda=1$ turns into
$$\mathcal{C}^1(x,q):=\{c\geq 0\text{ optional}:\ \bar{c}\vee q \text{ is } x\text{-admissible} \},$$
and for a functional $\mathbb{U}$ on $\mathcal{C}^1(x,q)$ satisfying \eqref{eq:monotone-functional} we obtain: if $c\in\mathcal{C}^1(x,q)$ is an optimizer then $\bar{c}\vee q\mathbbm{1}_{(0,\hat{T})}$ is also an optimizer and, in particular, it is an optimizer over all $x$-admissible processes $c'\in\mathcal{C}_{\text{inc}}$ with $c'_{0+}\geq q$. This argument allows us to completely embed an optimization problem over the set of $x$-admissible processes $c'\in\mathcal{C}_{\text{inc}}$ with $c'_{0+}\geq q$ into an optimization problem over the subset $\mathcal{C}^1(x,q)$ of non-negative optional processes. Finally, notice that the set $\mathcal{C}^1(x,q)$ contains all $x$-admissible non-negative optional $c$ such that for almost every $\omega\in\Omega$ the path $t\mapsto c_t(\omega)$ is non-decreasing and $c_{0+}(\omega)\geq q$. This holds because $c=\bar{c}\vee q$, $\mathbb{P}\times d\kappa-$a.e., for such $c$. Hence, it is completely legitimate to look at optimization of $\mathbb{U}$ satisfying \eqref{eq:monotone-functional} over all $x$-admissible non-decreasing optional processes that are (essentially) bounded from below by $q$ as optimization over $\mathcal{C}^1(x,q)$.

%%%%%%%%%%%%%%%%%%%%%%%%%%%%

\subsection{Formulation of the optimization problem}

Let $U=U(\omega,t,x):\Omega\times[0,\hat{T})\times[0,\infty)\to\mathbb{R}\cup\{-\infty\}$ be a \textit{utility stochastic field} satisfying the following conditions (same conditions as in \cite{mostovyi}):
\begin{assume}\label{ass:utility}
For every $(\omega,t)\in\Omega\times[0,\hat{T})$ the function $x\mapsto U(\omega,t,x)$ is strictly concave, increasing, continuously differentiable on $(0,\infty)$ and satisfies the Inada conditions:
$$\lim_{x\downarrow 0}U'(\omega,t,x)=+\infty\quad\text{and}\quad\lim_{x\to\infty}U'(\omega,t,x)=0,$$
where $U'$ denotes the partial derivative with respect to $x$. At $x=0$ we define, by continuity, $U(\omega,t,0):=\lim_{x\downarrow 0}U(\omega,t,x)$, this value may be $-\infty$. For every $x\geq 0$ the stochastic process $U(\cdot,\cdot,x)$ is optional.
\end{assume}

We consider the problem where an agent maximizes his expected utility of intertemporal consumption $c$ under the $x$-admissibility constraint \eqref{eq:x-admissible} and the drawdown constraint \eqref{cond:ddc-with-q}. This problem therefore can be seen as the optimization over the set $\mathcal{C}^\lambda(x,q)$ defined in \eqref{def:primal-dom}. The associated value function is
\begin{equation}\label{eq:primal-problem}
u(x,q):=\sup_{c\in\mathcal{C}^\lambda(x,q)}\mathbb{E}\left[\int_0^{\hat{T}} U(\omega,t,c_t)d\kappa_t\right],\quad (x,q)\in\mathcal{K},
\end{equation}
where an appropriate domain $\mathcal{K}\subseteq\mathbb{R}^2$ for $(x,q)$ is to be specified later.
Here we use the convention
$$\mathbb{E}\left[\int_0^{\hat{T}} U(\omega,t,c_t)d\kappa_t\right]:=-\infty\quad\text{if}\quad \mathbb{E}\left[\int_0^{\hat{T}} U^{-}(\omega,t,c_t)d\kappa_t\right]=+\infty,$$
where $W^{-}$ denotes the negative part of a stochastic field $W$.

Our goal is to develop duality arguments, based on \cite{mostovyi}, describing the solutions of the two-parameter optimization problem \eqref{eq:primal-problem} in similar fashion as it is done in \cite{hug-kramkov} for expected utility maximization problem with random endowments at maturity.

%%%%%%%%%%%%%%%%%%%%%%%%%%%%
%%%%%%%%%%%%%%%%%%%%%%%%%%%%

\section{Running essential supremum of a non-negative optional process}\label{sec:run-esssup}
In this section, we introduce a notion of running essential supremum. This is the appropriate generalization to measurable processes of what running maximum is for continuous processes. To the best of our knowledge, this process has not been previously considered in the literature.

Let $c$ be a non-negative optional process on $(\Omega,\mathcal{F},(\mathcal{F}_t)_{t\in[0,\hat{T})},\mathbb{P})$. Since $c$ is $\mathcal{F}\otimes\mathcal{B}([0,\hat{T}))$-measurable, every path of $c$ is Borel measurable. Therefore, the following pathwise definition of running essential supremum of $c$ makes sense.
\begin{defn}\label{def:running-max}
For every $\omega\in\Omega$, we define
$$\bar{c}_t(\omega):=\esssup_{s\in[0,t]}c_s(\omega)\in[0,\infty]\quad \text{ for } t\in(0,\hat{T}),\quad\text{ and } \bar{c}_0(\omega):=0,$$
where the essential supremum is taken with respect to the Lebesgue measure. We call $\bar{c}$ the \textit{running essential supremum} of $c$.
\end{defn}
In the sequel, we often denote by the bar over the process its running essential supremum without stating it explicitly.

\begin{rem} Let $\tilde{c}_t(\omega):=\sup_{s\in[0,t]}c_s(\omega)$ for $(\omega,t)\in\Omega\times[0,\hat{T})$. This is the smallest non-decreasing process satisfying $c_t(\omega)\leq\tilde{c}_t(\omega)$ for all $(\omega,t)\in\Omega\times[0,\hat{T})$. Clearly, $\tilde{c}\geq \bar{c}$ and if $c\geq \lambda(\tilde{c}\vee q)$ then $c$ satisfies \eqref{cond:ddc-with-q}. However, there is a reason why we consider $\bar{c}$ and not $\tilde{c}$ in the formulation of the drawdown constraint \eqref{cond:ddc-with-q}. An optimizer for \eqref{eq:primal-problem} is determined up to $\mathbb{P}\times d\kappa-$nullsets, therefore, our definition of the drawdown constraint should be indifferent to changes of $c$  on $\mathbb{P}\times d\kappa-$nullsets. Hence, it would be natural to adopt a definition of running maximum which gives out the same process, up to indistinguishability, for any two processes that are equal $\mathbb{P}\times d\kappa-$a.e. This holds for $\bar{c}$ (Proposition \ref{prop:mon-of-sup}) but not for $\tilde{c}$. It also seems reasonable from an economic perspective to have a definition of running maximum that is unaffected by a jump of consumption rate $c$ at a single point, since this jump does not even affect the expected utility functional eventually.
\end{rem}
We will give an alternative definition of the running essential supremum, which is equivalent and makes it easier to prove certain properties of $\bar{c}$. First, we introduce the notion of the essential debut of level $l\geq 0$ by process $c$.

%%%%%%%%%%%%%%%%%%%%%%%%%%%%

\subsection{Essential debut}\label{sec:ess-debut}

Let $l\in[0,\infty)$. Define a stochastic process
$$X_t^l(\omega)=\int_0^t\mathbbm{1}_{\{c_s(\omega)\geq l\}}ds,\quad t\in[0,\hat{T}),$$
which is non-negative, non-decreasing, continuous, and adapted. Define
$$\tau^l(\omega)=\inf\{t\in[0,\hat{T}): X_t^l(\omega)>0\},$$
the hitting time of $(0,\infty)$ by the process $X^l$, where by convention the infimum of an empty set is $\hat{T}$. By construction, $\tau^l$ can also be characterized as follows:
$$\tau^l(\omega)=\inf\{t\in[0,\hat{T}): \vert \left\{s\in[0,t]: c_s(\omega)\geq l\right\}\vert>0\},$$
where $\vert\cdot\vert$ denotes the Lebesgue measure of a set. That is, $\tau^l$ is the \textit{essential debut} (see \cite{dellacherie-meyerA}, Chapter IV, p.108) of the set $\{(\omega,t): c_t(\omega)\geq l\}$: given $\omega\in\Omega$, the set $\{s\in[0,t]: c_s(\omega)\geq l\}$ has a positive Lebesgue measure for every $t\in(\tau^l(\omega),\hat{T})$, and $\tau^l(\omega)$ is the smallest time satisfying this property. Clearly, the essential debut $\tau^l$ is no smaller than the hitting time $\inf\{t\in[0,\hat{T}):c_t(\omega)\geq l\}$, but it could be strictly larger if the process $c_t$ does not spend a strictly positive amount of time (in the sense of the Lebesgue measure) at or above level $l$ right after hitting it.
\begin{prop}\label{prop:stopping-time}
For every $l\geq 0$, $\tau^l$ is a stopping time.
\end{prop}
\begin{proof}
Since $\tau^l$ is the hitting time of the open set $(0,\infty)$ by the continuous adapted process $X^l$ and since the filtration $\left(\mathcal{F}_t\right)_{t\in[0,\hat{T})}$ is right-continuous, $\tau^l$ is a stopping time (see, for example, \cite{karatzas-shreve-book}, p.7, 2.6 Problem).
\end{proof}

%%%%%%%%%%%%%%%%%%%%%%%%%%%%

\subsection{Alternative characterization of $\bar{c}_t$}
In the sequel, we often omit writing $\omega$ in pathwise statements.
\begin{prop}\label{prop:second-def}
The definition of the running essential supremum $\bar{c}_t$ in Definition \ref{def:running-max} is equivalent to the following:
$$\bar{c}_t=\sup\{l\geq 0: \tau^l<t\}=\inf\{l\geq 0: \tau^l\geq t\}.$$
\end{prop}
Note that the function $l\mapsto\tau^l(\omega)$ is non-decreasing for every $\omega\in\Omega$, hence Proposition~\ref{prop:second-def} says that $t\mapsto\bar{c}_t(\omega)$ is a generalized inverse of the non-decreasing function $l\mapsto\tau^l(\omega)$.
\begin{proof}
The statement is obvious for $t=0$, so let us fix $\omega\in\Omega$ and $t\in(0,\hat{T})$. Since $l\mapsto\tau^l(\omega)$ is a non-decreasing function, the second equality, between supremum and infimum, holds and we can denote $\theta:=\sup\{l\geq 0: \tau^l<t\}=\inf\{l\geq 0: \tau^l\geq t\}$.

For every $0\leq\nu<\theta$, $\tau^{\nu+\varepsilon}<t$ for all $\varepsilon>0$ small enough, therefore $$\vert\left\{s\in[0,t]: c_s>\nu\right\}\vert\geq\vert\left\{s\in[0,t]: c_s\geq\nu+\varepsilon\right\}\vert>0$$ for an $\varepsilon$ small enough, and hence $\nu<\esssup_{s\in[0,t]}c_s$ by the definition of essential supremum. This proves that $\theta\leq\esssup_{s\in[0,t]}c_s$. On the other hand, for every $\nu>\theta$, $\tau^\nu\geq t$ and therefore
$\vert \left\{s\in[0,t]: c_s\geq \nu\right\}\vert=0$. This implies that $\vert \{s\in[0,t]: c_s> \theta\}\vert=0$, i.e., $\theta\geq\esssup_{s\in[0,t]}c_s$,
and that the two definitions of $\bar{c}_t$, as the essential supremum and as the generalized inverse, coincide.
\end{proof}

%%%%%%%%%%%%%%%%%%%%%%%%%%%%

\subsection{Properties of $\bar{c}_t$}
Clearly, $t\mapsto\bar{c}_t(\omega)$ is non-decreasing. In the following two propositions we show that the process $\bar{c}$ is left-continuous, predictable, and $c_t(\omega)\leq\bar{c}_t(\omega)$ for every $\omega\in\Omega$, Lebesgue-a.e. in $t$. 

\begin{prop}\label{prop:run-sup-properties}
The running essential supremum $\bar{c}$ of a non-negative optional process $c$ is left-continuous and predictable.
\end{prop}
In particular, this proposition shows that $t\mapsto\bar{c}_t(\omega)$ is the \textit{left-continuous generalized inverse} of the non-decreasing function $l\mapsto\tau^l(\omega)$ (cf.~definition of the right-continuous generalized inverse, for example, in \cite{karatzas-shreve-book}, p.174, 4.5 Problem).
\begin{proof}
To prove left-continuity, let us fix a path $\omega\in\Omega$ and $t\in(0,\hat{T})$. Since $s\mapsto\bar{c}_s$ is non-decreasing, we need to show that for every $\varepsilon>0$ there exists $\delta>0$ such that $\bar{c}_{t-\delta}\geq \bar{c}_t-\varepsilon$. Indeed, by the definition of $\bar{c}_t$,
$$m:=\vert\{s\in[0,t]: c_s\geq \bar{c}_t-\varepsilon\}\vert>0,$$
and, in particular, $\vert\{s\in[0,t-m/2]: c_s\geq \bar{c}_t-\varepsilon\}\vert\geq m/2>0$. Hence $\bar{c}_{t-m/2}\geq \bar{c}_t-\varepsilon$.

To prove predictability, we write for every $\nu>0$ and a sequence of numbers $\varepsilon_n\downarrow 0$,
\begin{equation}\label{eq:predictability}
\left\{(\omega,t): \bar{c}_t(\omega)<\nu\right\}=\cup_{n\geq 1}\left\{(\omega,t): \tau^{\nu-\varepsilon_n}(\omega)\geq t\right\}.
\end{equation}
The inclusion ``$\subseteq$'' in \eqref{eq:predictability} holds because if $\bar{c}_t(\omega)<\nu$ then for $n$ large enough (i.e., for $\varepsilon_n$ small enough) $\tau^{\nu-\varepsilon_n}(\omega)\geq t$ thanks to the characterization of $\bar{c}_t$ as the generalized inverse of $\tau^k$. Conversely, if $\tau^{\nu-\varepsilon}(\omega)\geq t$ for some $\varepsilon>0$ then $\bar{c}_t(\omega)\leq\nu-\varepsilon<\nu$.

Note that for every $\varepsilon>0$ the set $\left\{(\omega,t): \tau^{\nu-\varepsilon}(\omega)\geq t\right\}$ can be rewritten as $$\left\{(\omega,t): \mathbbm{1}_{(\tau^{\nu-\varepsilon}(\omega),\hat{T})}(t)=0\right\}.$$
Since $\tau^{\nu-\varepsilon}$ is a stopping time, the indicator function $\mathbbm{1}_{(\tau^{\nu-\varepsilon},\hat{T})}$ is measurable with respect to the predictable sigma-algebra. Hence, the right-hand side of \eqref{eq:predictability} is a countable union of predictable sets, which implies that the set $\{\bar{c}_t(\omega)<\nu\}$ is predictable for every $\nu>0$ and therefore $\bar{c}_t$ is a predictable stochastic process.
\end{proof}

\begin{prop}\label{prop:increasing-dom}
The running essential supremum $\bar{c}$ of a non-negative optional process $c$ satisfies for every $\omega\in\Omega$ the inequality
$$c_t(\omega)\leq\bar{c}_t(\omega)\quad \text{for Lebesgue almost every } t\in[0,\hat{T}).$$
\end{prop}
\begin{proof}
Fix a path $\omega\in\Omega$. The statement would follow if we show that for every $\varepsilon>0$ the set $A^\varepsilon:=\{t\in[0,\hat{T}):\bar{c}_t\leq c_t-\varepsilon\}$ has zero Lebesgue measure. Assume that $\vert A^\varepsilon\vert>0$ for some $\varepsilon>0$. Then there exists $t_1\in A^\varepsilon$ such that $\vert A^\varepsilon\cap [0,t_1]\vert>0$. Otherwise, the set $A^\varepsilon\cap [0,\sup A^\varepsilon)\subseteq\cup_{k\geq 1}(A^\varepsilon\cap [0,s_k])$, where $s_k\in A^\varepsilon$ and $s_k\uparrow\sup A^\varepsilon$ as $k\to\infty$, has measure zero, a contradiction. The fact that $t_1\in A^\varepsilon$ and the definition of $\bar{c}_{t_1}$ as the essential supremum imply that the set
$$B_1:=\{s\in[0,t_1]: c_s>c_{t_1}-\varepsilon\}$$
has Lebesgue measure zero. Therefore, the set $A^\varepsilon\cap [0,t_1]\setminus B_1$ still has a strictly positive Lebesgue measure, and $c_s\leq c_{t_1}-\varepsilon$ on $[0,t_1]\setminus B_1$. Now we can choose $t_2\in A^\varepsilon\cap [0,t_1]\setminus B_1$ such that $\vert A^\varepsilon\cap [0,t_2]\setminus B_1\vert>0$ and define the set
$$B_2:=\{s\in[0,t_2]: c_s>c_{t_2}-\varepsilon\}.$$
In the same way, we obtain that $\vert B_2\vert=0$, $\vert A^\varepsilon\cap [0,t_2]\setminus (B_1\cup B_2)\vert>0$, and $c_s\leq c_{t_2}-\varepsilon\leq c_{t_1}-2\varepsilon$ on $[0,t_2]\setminus (B_1\cup B_2)$. Continuining this procedure analogously, after $n$ steps we obtain $n$ Lebesgue-negligible sets $B_1,...,B_n$ and a number $t_n\in(0,\hat{T})$ such that $$\vert A^\varepsilon\cap [0,t_n]\setminus (B_1\cup ...\cup B_n)\vert>0$$ and $c_s\leq c_{t_1}-n\varepsilon$ on $[0,t_n]\setminus (B_1\cup...\cup B_n)$. This cannot be repeated infinitely many times since the process $c$ is non-negative and $c_{t_1}<\infty$, and the contradiction implies that $\vert A^\varepsilon\vert=0$.
\end{proof}

By Proposition \ref{prop:run-sup-properties}, the running essential supremum $\bar{c}$ of every non-negative optional process $c$ belongs to $\mathcal{C}_{\text{inc}}$ (see Definition \ref{def:C-0-class}). Note that since two left-continuous stochastic processes on $(\Omega,[0,\hat{T}))$ are equal $\mathbb{P}\times\text{Leb}([0,\hat{T}))-$a.e. if and only if they are indistinguishable, the same is true for any two processes in $\mathcal{C}_{\text{inc}}$. Moreover, the following property holds for running essential suprema:
\begin{prop}\label{prop:mon-of-sup}
If $c^1$, $c^2$ are two non-negative optional processes such that $c^1\leq c^2$, $\mathbb{P}\times\text{Leb}([0,\hat{T}))-$a.e., then
$$\mathbb{P}\left(\{\omega: \bar{c}^1_t(\omega)\leq\bar{c}^2_t(\omega)\text{ for every }t\in[0,\hat{T})\}\right)=1.$$
In particular, if $c^1=c^2$, $\mathbb{P}\times\text{Leb}([0,\hat{T}))-$a.e., then $\bar{c}^1$ and $\bar{c}^2$ are indistinguishable.
\end{prop}
\begin{proof}
If $\bar{c}^1_t(\omega)>\bar{c}^2_t(\omega)$ for some $t\in(0,\hat{T})$ then $\vert s\in [0,\hat{T}): c^1_s(\omega)>c^2_s(\omega)\vert>0$. However, by Fubini's theorem, since $c^1\leq c^2$, $\mathbb{P}\times\text{Leb}([0,\hat{T}))-$a.e.,
$$\mathbb{P}\left(\{\omega: \vert s\in[0,\hat{T}): c^1_s(\omega)>c^2_s(\omega)\vert>0\}\right)=0,$$
proving the first assertion. The second assertion easily follows from the first.
\end{proof}
The following proposition states that $\bar{c}$ is the smallest element of $\mathcal{C}_{\text{inc}}$ dominating $c$.
\begin{prop}\label{prop:minimality-of-c-bar}
If $c$ is a non-negative optional process and $c'\in\mathcal{C}_{\text{inc}}$ is such that $c\leq c'$, $\mathbb{P}\times\text{Leb}([0,\hat{T}))-$a.e., then $\bar{c}_t\leq c'_t$ for all $t\in [0,\hat{T})$, $\mathbb{P}-$a.e.
\end{prop}
\begin{proof}
Assume that $\bar{c}_t(\omega)>c'_t(\omega)$ for some $\omega\in\Omega$ and $t\in(0,\hat{T})$. Then, since $c'$ is non-decreasing and by the definition of running essential supremum,
\begin{equation}\label{eq:ineq-a-e}
\vert\{s\in[0,\hat{T}): c_s(\omega)>c'_s(\omega)\}\vert>0.
\end{equation}
Since $c\leq c'$, $\mathbb{P}\times\text{Leb}([0,\hat{T}))-$a.e., by Fubini's theorem, the set of $\omega$'s satisfying \eqref{eq:ineq-a-e} has measure zero. Hence,
$\mathbb{P}\left(\{\omega: \bar{c}_t(\omega)\leq c'_t(\omega)\text{ for all }t\in [0,\hat{T})\}\right)=1.$
\end{proof}

%%%%%%%%%%%%%%%%%%%%%%%%%%%%

\subsection{Integration with respect to $c\in\mathcal{C}_{\text{inc}}$}

Fix a $[0,\infty)$-valued process $c\in\mathcal{C}_{\text{inc}}$. For every $\omega\in\Omega$, $c$ uniquely defines a non-negative Borel measure $dc_t(\omega)$ on $[0,\hat{T})$ via
$$dc_t(\omega)([0,t))=c_{t}(\omega),\quad t\in(0,\hat{T}).$$
Note that if $c(\omega)$ has a jump at $t$, then $dc_t(\omega)(\{t\})=c_{t+}-c_t$. This is the same random measure that comes from the non-decreasing c\`adl\`ag adapted process 
$$c'_t:=c_{t+}=\lim_{\varepsilon\downarrow 0}c_{t+\varepsilon}\quad\text{for } t\in[0,\hat{T}),\quad\quad c'_{0-}:=0.$$
Integration with respect to non-decreasing c\`adl\`ag adapted processes, being a generalization of the Lebesgue-Stieltjes integration to stochastic processes, is a classical topic in the general theory of stochastic processes (see, for example, \cite{dellacherie-meyerB}, Chapter VI, no. 51--58) and in the next proposition we list several properties of integration with respect to $dc_t=dc'_t$ that will be extensively used in this paper.

Throughout the paper, we adopt the following conventions:
\begin{enumerate}
\item The integral $\int_0^{\hat{T}}$ is understood as $\int_{[0,\hat{T})}$, i.e., $0$ is included in the domain of the integration.
\item For two $\mathcal{F}\otimes\mathcal{B}([0,\hat{T}))$-measurable processes $D^{1,2}:\Omega\times[0,\hat{T})\to[0,\infty)$ we say that $$D^1\leq D^2\quad \text{up to indistinguishability}$$
if $\mathbb{P}\left(\{\omega: D^1_t(\omega)\leq D^2_t(\omega)\text{ for all }t\in[0,\hat{T})\}\right)=1$. By the Optional Section Theorem, for optional processes $D^{1,2}$ this condition is equivalent to
$D^1_T\mathbbm{1}_{\{T<\hat{T}\}}\leq D^2_T\mathbbm{1}_{\{T<\hat{T}\}}$, $\mathbb{P}-$a.e., for every stopping time $T:\Omega\to[0,\hat{T}]$.
 \end{enumerate}
\begin{prop}
Let $c\in\mathcal{C}_{\text{inc}}$ be $[0,\infty)$-valued. Then
\begin{enumerate}
\item For every $D:\Omega\times[0,\hat{T})\to[0,\infty)$ that is $\mathcal{F}\otimes\mathcal{B}([0,\hat{T}))$-measurable, we have
\begin{equation}\label{eq:int-optional-proj}
\mathbb{E}\left[\int_0^{\hat{T}} D_tdc_t\right]=\mathbb{E}\left[\int_0^{\hat{T}}{}^oD_tdc_t\right],
\end{equation}
where ${}^oD$ denotes the optional projection of $D$.
\item For every $f:\Omega\times[0,\hat{T})\to[0,\infty)$ that is $\mathcal{F}\otimes\mathcal{B}([0,{\hat{T}}))$-measurable, we have
\begin{equation}\label{eq:IBP}
\mathbb{E}\left[\int_0^{\hat{T}} f_t c_t dt\right]=\mathbb{E}\left[\int_0^{\hat{T}} \left(\int_s^{\hat{T}} f_tdt\right)dc_s\right].
\end{equation}
\item If two optional processes $D^{1,2}:\Omega\times[0,{\hat{T}})\to[0,\infty)$ satisfy $D^1_T\mathbbm{1}_{\{T<\hat{T}\}}\leq D^2_T\mathbbm{1}_{\{T<\hat{T}\}}$, $\mathbb{P}-$a.e., for every stopping time $T\in[0,\hat{T}]$, then
\begin{equation}\label{eq:comparison-for-int}
\mathbb{E}\left[\int_0^{\hat{T}} D^1_tdc_t\right]\leq\mathbb{E}\left[\int_0^{\hat{T}} D^2_tdc_t\right].
\end{equation}
\end{enumerate}
\end{prop}
\begin{proof}
Item 1 is proved in \cite{dellacherie-meyerB}, Chapter VI, no. 57. Item 2 follows from
$$\mathbb{E}\left[\int_0^{\hat{T}} f_t c_t dt\right]=\mathbb{E}\left[\int_0^{\hat{T}} f_t \int_{[0,t)}dc_s dt\right]$$
and Fubini's theorem. Item 3 follows since the Optional Section Theorem implies that $D^1\leq D^2$ up to indistinguishability.
\end{proof}
We will also need the following definition and lemma on several occasions.
\begin{defn}\label{def:pt-of-increase}
For $c\in\mathcal{C}_{\text{inc}}$ and $\omega\in\Omega$, we write $dc_t(\omega)>0$ (or simply $dc_t>0$) and call $t\in[0,\hat{T})$ a \textit{point of increase} of $c(\omega)$ if $c_{s}(\omega)>c_t(\omega)$ for all $s\in(t,\hat{T})$.
\end{defn}
\begin{lem}\label{lem:pt-of-increase}
For a non-negative optional right-continuous process $D$ and for $c\in\mathcal{C}_\text{inc}$, the following are equivalent:
\begin{enumerate}
\item $D_t=0$, $\mathbb{P}\times dc_t-$almost everywhere on $\Omega\times[0,{\hat{T}})$.
\item $\mathbb{P}-$almost everywhere, $D_t= 0$ for all $t\in[0,{\hat{T}})$ such that $dc_t>0$.
\end{enumerate}
\end{lem}
\begin{proof}
$``1.\Rightarrow 2."$
Assume that for some $\omega\in\Omega$ there exists $t\in[0,{\hat{T}})$ such that $D_t(\omega)> 0$ and $dc_t(\omega)>0$. Then $D_s(\omega)>0$ for all $s\in[t,t+\varepsilon]$ for $\varepsilon>0$ small enough by right-continuity of $D$, but also $dc(\omega)([t,t+\varepsilon])>0$. Hence $\int_0^{\hat{T}} D_t(\omega)dc_t(\omega)>0$, which is only possible on a $\mathbb{P}-$nullset by 1.

$``2.\Rightarrow 1."$
For $l\geq 0$, define a stopping time $T_l:=\inf\{t\in[0,{\hat{T}}): c_t>l\}$, where the infimum of an empty set is taken to be ${\hat{T}}$. This is indeed a stopping time by the right-continuity of the filtration. Since $c$ is left-continuous and non-decreasing, $c_{T_l}\leq l<c_{s}$ for any $s\in(T_l,{\hat{T}})$, hence $T_l$ is either ${\hat{T}}$, or a point of increase of $c$. Therefore, by 2., $D_{T_l}\mathbbm{1}_{\{T_l<{\hat{T}}\}}=0$ for every $l\geq0$, $\mathbb{P}-$almost surely. The time-change formula (55.1) in \cite{dellacherie-meyerB}, Chapter VI, when applied to the  c\`adl\`ag version $c'_t$ of $c_t$, states that
$$\mathbb{E}\left[\int_0^{\hat{T}} D_tdc_t\right]=\mathbb{E}\left[\int_0^\infty D_{T_l}\mathbbm{1}_{\{T_l<{\hat{T}}\}}dl\right].$$
Since the right-hand side is zero, it follows that $D_t=0$, $\mathbb{P}\times dc_t-$almost everywhere.
\end{proof}

%%%%%%%%%%%%%%%%%%%%%%%%%%%%
%%%%%%%%%%%%%%%%%%%%%%%%%%%%
%%%%%%%%%%%%%%%%%%%%%%%%%%%%

\section{Domains}\label{sec:domains}

%%%%%%%%%%%%%%%%%%%%%%%%%%%%

\subsection{The primal domain}
Recall the definition $\mathcal{C}^\lambda=\mathcal{C}^\lambda(1,0)=\{c\geq 0 \text{ optional}:c\vee\lambda\bar{c}\text{ is 1-admissible}\}$ for $\lambda\in[0,1]$. The following proposition is important for being able to apply the results from \cite{mostovyi} to our optimization problem \eqref{eq:primal-problem}.
\begin{prop}\label{prop:C-bipolar}
For $\lambda\in[0,1]$, the set $\mathcal{C}^\lambda$ is convex, solid, and closed with respect to convergence in (finite) measure $\mathbb{P}\times d\kappa$ on $\Omega\times[0,\hat{T})$.
\end{prop}
\begin{proof} Solidity is obvious. To prove convexity, let $c^1,c^2\in\mathcal{C}^\lambda$. Since $c^1\vee\lambda\bar{c}^1,c^2\vee\lambda\bar{c}^2$ are $1$-admissible, $\theta(c^1\vee\lambda\bar{c}^1)+(1-\theta)(c^2\vee\lambda\bar{c}^2)$ is also $1$-admissible for $\theta\in[0,1]$. The running essential supremum of $\theta c^1+(1-\theta)c^2$ is bounded from above by $\theta\bar{c}^1+(1-\theta)\bar{c}^2$. Hence, both the process $\theta c^1+(1-\theta)c^2$ and $\lambda$ times its running essential supremum are bounded from above by the $1$-admissible process $\theta(c^1\vee\lambda\bar{c}^1)+(1-\theta)(c^2\vee\lambda\bar{c}^2)$, implying that $\theta c^1+(1-\theta)c^2\in\mathcal{C}^\lambda$ and proving convexity.

To prove that $\mathcal{C}^\lambda$ is closed, let $(c^n)_{n\geq1}\subseteq\mathcal{C}^\lambda$ converge to $c$ almost everywhere. After replacing $c^n$ by $\tilde{c}^n:=\inf_{k\geq n}c^k\in\mathcal{C}^\lambda$ and using Proposition \ref{prop:mon-of-sup}, we can assume without loss of generality that the sequence $c^n$ increases to $c$ pointwise on $\Omega\times[0,\hat{T})$. Then it is easy to check directly from Definition \ref{def:running-max} that $\bar{c}^n\uparrow \bar{c}$ pointwise, therefore $c^n\vee\lambda\bar{c}^n\uparrow c\vee\lambda\bar{c}$ pointwise as well, and, by monotone convergence applied to the characterization \eqref{eq:x-admissible} of $x$-admissiblity, $c\vee\lambda\bar{c}$ is $1$-admissible, implying $c\in\mathcal{C}^\lambda$.
\end{proof}

%%%%%%%%%%%%%%%%%%%%%%%%%%%%

\subsection{Chronological ordering}
The following ordering will be crucial in our definition and characterization of the dual domains. This ordering appears implicitly in \cite{BK} through their definition (11). Borrowing this idea, we solve the singular control problem of \cite{BK} type, corresponding to $\lambda=1$ and $q=0$ in our framework, in what seems to be a more direct way (i.e., with a more straightforward definition of the dual domain and of the dual optimization problem; see Remark \ref{rem:duality-thm}(iii)) and extend the method to constraints given by parameters $\lambda\in(0,1)$ and $q>0$.
\begin{defn}\label{defn:chron-ord}
For a stochastic clock $\kappa$ satisfying Assumption \ref{ass:clock}, we define a \textit{chronological ordering} $\preceq$ on the set of non-negative optional processes on $\Omega\times[0,\hat{T})$ as follows: $\tilde{\delta}\preceq\delta$ if and only if
\begin{equation}\label{def:chron-ord}
^o\left(\int_.^{\hat{T}} \tilde{\delta}d\kappa\right)\leq{}^o\left(\int_.^{\hat{T}} \delta d\kappa\right)\quad \text{up to indistinguishability}.
\end{equation}
\end{defn}
\begin{rem}\label{rem:chron-ord-def}
\begin{enumerate}
\item[(i)] By the Optional Section Theorem and by the definition of optional projection, condition \eqref{def:chron-ord} is equivalent to saying that for every stopping time $T\in[0,\hat{T}]$,
\begin{equation}\label{def:chron-alternative}
\mathbb{E}\left[\left(\int_T^{\hat{T}}\tilde\delta d\kappa\right) \mathbbm{1}_{\{T<{\hat{T}}\}}\Big\vert\mathcal{F}_T\right]\leq \mathbb{E}\left[\left(\int_T^{\hat{T}}\delta d\kappa\right) \mathbbm{1}_{\{T<{\hat{T}}\}}\Big\vert\mathcal{F}_T\right],\quad \mathbb{P}-\text{a.e.}
\end{equation}
For $B\in\mathcal{F}_T$, $T_B(\omega):=\begin{cases}T(\omega)\ \text{if }\omega\in B,\\ {\hat{T}} \text{ otherwise},\end{cases}$ is a stopping time and $\mathbbm{1}_{\{T<{\hat{T}}\}}\cdot\mathbbm{1}_B=\mathbbm{1}_{\{T_B<{\hat{T}}\}}$, therefore, \eqref{def:chron-ord} and \eqref{def:chron-alternative} are equivalent to the following:
\begin{equation}\label{eq:chron-ord-characterization}
\mathbb{E}\left[\int_T^{\hat{T}}\tilde\delta d\kappa\right]\leq \mathbb{E}\left[\int_T^{\hat{T}}\delta d\kappa\right]\quad\text{for every stopping time } T\in[0,{\hat{T}}]
\end{equation}
(where we discarded the indicator function $\mathbbm{1}_{\{T<{\hat{T}}\}}$ because $\int_{\hat{T}}^{\hat{T}}\square d\kappa=0$).
\item[(ii)] If $\tilde{\delta}\leq\delta$, $\mathbb{P}\times d\kappa-$a.e., then $\left(\int_T^{\hat{T}}\tilde{\delta}d\kappa\right)\mathbbm{1}_{\{T<{\hat{T}}\}}\leq\left(\int_T^{\hat{T}}\delta d\kappa\right)\mathbbm{1}_{\{T<{\hat{T}}\}}$, $\mathbb{P}-$a.e., for every stopping time $T\in[0,{\hat{T}}]$, implying that $\tilde{\delta}\preceq\delta$.
\end{enumerate}
\end{rem}
Intuitively, all the processes obtained from $\delta$ by ``moving some of the mass of $\delta d\kappa$ to earlier times'', possibly removing some of the mass, and taking the optional projection afterwards, are $\preceq$-dominated by $\delta$, hence the name ``chronological''. This is best seen in the deterministic case, i.e., on a trivial filtered probability space, when the definition of $\preceq$ turns into comparison of mass of $\tilde\delta d\kappa$ and $\delta d\kappa$ on each of the intervals $(t,{\hat{T}})$ for $t\in[0,{\hat{T}})$. The most subtle application of this intuition of moving mass to earlier in time is in Step 1 of the proof of Lemma \ref{lem:important} below.

The following is the remarkable property of the chronological ordering:
\begin{prop}\label{prop:remarkable-property}
For two non-negative optional processes $\tilde{\delta}$ and $\delta$,
\begin{equation}\label{eq:monotonicity-chron-ord}
\tilde{\delta}\preceq\delta\quad \Leftrightarrow\quad \langle c,\tilde\delta\rangle\leq\langle c,\delta \rangle\text{ for all }c\in\mathcal{C}_{\text{inc}}.
\end{equation}
\end{prop}
In the informal language from above, the most important message of this proposition reduces to the obvious statement that ``moving some of the mass of $\delta d\kappa$ to earlier times can only make the integral of an increasing process with respect to $\delta d\kappa$ smaller''.
\begin{proof}
To show ``$\Leftarrow$'', take $c:=\mathbbm{1}_{(T,{\hat{T}})}$ for every stopping time $T$ and apply the equivalent characterization \eqref{eq:chron-ord-characterization} of $\preceq$. For ``$\Rightarrow$'', when $c$ is finite-valued, we have
\begin{equation*}
\begin{aligned}
\langle c,\tilde\delta\rangle&=\mathbb{E}\left[\int_0^{\hat{T}} c\tilde{\delta} d\kappa\right]\overset{\eqref{eq:IBP}}{=}\mathbb{E}\left[\int_0^{\hat{T}}\left(\int_t^{\hat{T}} \tilde{\delta} d\kappa\right)dc_t\right]\overset{\eqref{eq:int-optional-proj}}{=}\mathbb{E}\left[\int_0^{\hat{T}} {}^o\left(\int_.^{\hat{T}} \tilde{\delta} d\kappa\right)_tdc_t\right]\\
&\overset{\eqref{eq:comparison-for-int}}{\leq}\mathbb{E}\left[\int_0^{\hat{T}} {}^o\left(\int_.^{\hat{T}} \delta d\kappa\right)_tdc_t\right]\overset{\eqref{eq:int-optional-proj},\eqref{eq:IBP}}{=}\mathbb{E}\left[\int_0^{\hat{T}} c\delta d\kappa\right]=\langle c,\delta\rangle.
\end{aligned}
\end{equation*}
An arbitrary $c\in\mathcal{C}_{\text{inc}}$ can be approximated by finite-valued $c\wedge N\in\mathcal{C}_{\text{inc}}$ as $N\to\infty$ and the right-hand side of \eqref{eq:monotonicity-chron-ord} follows by monotone convergence.
\end{proof}
Now we introduce an ordering $\preceq_\lambda$ that will be used to characterize dual domains for $\lambda<1$.
\begin{defn}\label{def:lambda-ordering}
For $\lambda\in[0,1]$ and $\delta,\tilde\delta\geq 0$ optional, we define $\tilde\delta\preceq_\lambda\delta$ to hold if and only if there exists a representation
$\delta=\delta^1+\delta^2$, $\tilde\delta=\tilde\delta^1+\tilde\delta^2$ with $\delta^i,\tilde\delta^i\geq 0$ optional, $i=1,2$, such that
$$\tilde\delta^1\leq\delta^1,\ \mathbb{P}\times d\kappa-\text{a.e.},\quad\text{and}\quad\tilde\delta^2\preceq\lambda\delta^2.$$
\end{defn}
\begin{rem}
Note that $\preceq_1$ coincides with $\preceq$  by Remark \ref{rem:chron-ord-def}(ii) and that $\tilde\delta\preceq_0\delta$ if and only if $\tilde\delta\leq\delta$, $\mathbb{P}\times d\kappa-$a.e. Moreover, $\tilde\delta\preceq_{\lambda_1}\delta$ implies $\tilde\delta\preceq_{\lambda_2}\delta$ for $0\leq\lambda_1\leq\lambda_2\leq 1$. Hence, $\preceq_\lambda$ for $\lambda\in[0,1]$ can be seen as a family of orderings obtained by interpolation between ``$\leq$ up to $\mathbb{P}\times d\kappa-$nullsets'' and $\preceq$.
\end{rem}
Informally, $\tilde\delta\preceq_\lambda\delta$ if, up to taking optional projections, $\tilde\delta$ is obtained from $\delta$ by leaving some mass of $\delta d\kappa$ where it is, ``moving some of its mass to earlier time'' while also multiplying it by $\lambda$ (i.e., the $(1-\lambda)$ fraction of mass gets lost), and removing the rest.

An alternative characterization of $\preceq_\lambda$ is given in the following proposition.
\begin{prop}\label{prop:lambda-ord-alternative}
For $\lambda\in[0,1]$ and $\delta,\tilde\delta\geq 0$ optional,
$$\tilde\delta\preceq_\lambda\delta\quad \Leftrightarrow \quad (\tilde\delta-\delta)\vee 0\preceq \lambda(\delta-\tilde\delta)\vee 0.$$
\end{prop}
\begin{proof}
If $(\tilde\delta-\delta)\vee 0\preceq \lambda(\delta-\tilde\delta)\vee 0$ then the processes $\delta^1:=\tilde\delta^1:=\tilde\delta\wedge\delta$, $\delta^2:=(\delta-\tilde\delta)\vee 0$, and $\tilde\delta^2:=(\tilde\delta-\delta)\vee 0$ satisfy the conditions of Definition \ref{def:lambda-ordering}. Conversely, for any decomposition as in Definition \ref{def:lambda-ordering}, it holds that $\tilde\delta^1\leq \tilde\delta$ and $\tilde\delta^1\leq\delta^1\leq\delta$, hence $\tilde\delta^1\leq \tilde\delta\wedge\delta$ and $\tilde\delta^2\geq (\tilde\delta-\delta)\vee 0$. We can then write
\begin{equation*}
\begin{aligned}
(\tilde\delta-\delta)\vee 0&=\tilde\delta^2-\left[(\tilde\delta\wedge\delta)-\tilde\delta^1\right]\leq \tilde\delta^2- \lambda\left\{\left[(\tilde\delta\wedge\delta)-\tilde\delta^1\right]\wedge \delta^2\right\}\\
&\preceq \lambda\delta^2- \lambda\left\{\left[(\tilde\delta\wedge\delta)-\tilde\delta^1\right]\wedge \delta^2\right\}=\lambda\left[\delta^2+\tilde\delta^1-(\tilde\delta\wedge\delta)\right]\vee 0\\
&\leq \lambda\left[\delta-(\tilde\delta\wedge\delta)\right]\vee 0=\lambda(\delta-\tilde\delta)\vee 0,
\end{aligned}
\end{equation*}
proving the claim.
\end{proof}
Hence, $\tilde\delta\preceq_\lambda\delta$ if and only if the conditions of Definition \ref{def:lambda-ordering} hold with $\delta^1=\tilde\delta^1=\tilde\delta\wedge\delta$, i.e., in the case when we choose not to move all the mass that is possible not to move if we want to obtain $\tilde\delta$ from $\delta$.

The following proposition is an extension of Proposition \ref{prop:remarkable-property} to $\preceq_\lambda$.
\begin{prop}
For $\lambda\in[0,1]$ and $\tilde\delta,\delta\geq 0$ optional,
\begin{equation}\label{eq:monotonicity-lambda-ord}
\tilde{\delta}\preceq_\lambda\delta\quad \Leftrightarrow\quad \langle c,\tilde{\delta}\rangle\leq\langle c,\delta \rangle\text{ for all }c\text{ satisfying \eqref{cond:ddc}}.
\end{equation}
\end{prop}
\begin{proof}
For ``$\Rightarrow$'', let $\delta^i,\tilde\delta^i$, $i=1,2$, be as in Definition \ref{def:lambda-ordering}. Then for all $c$ satisfying \eqref{cond:ddc} we have
\begin{equation*}
\langle c,\tilde\delta^2 \rangle\leq\langle \bar{c},\tilde\delta^2 \rangle\overset{\eqref{eq:monotonicity-chron-ord}}{\leq}\langle \bar{c},\lambda\delta^2 \rangle=\langle\lambda\bar{c},\delta^2\rangle\leq\langle c,\delta^2 \rangle.
\end{equation*}
The right-hand side of \eqref{eq:monotonicity-lambda-ord} now follows by adding the above inequality and $\langle c,\tilde\delta^1 \rangle\leq\langle c,\delta^1 \rangle$. To prove ``$\Leftarrow$'', we define $\delta^1=\tilde\delta^1=\tilde\delta\wedge\delta$, $\delta^2=(\delta-\tilde\delta)\vee 0$, and $\tilde\delta^2=(\tilde\delta-\delta)\vee 0$, as in Proposition \ref{prop:lambda-ord-alternative}, and show that $\tilde\delta^2\preceq\lambda\delta^2$ by testing the right-hand side of \eqref{eq:monotonicity-lambda-ord} with suitable $c$'s. The right-hand side of \eqref{eq:monotonicity-lambda-ord} implies that
\begin{equation}\label{eq:lambda-ord-aux}
\langle c,\tilde\delta^2\rangle\leq\langle c,\delta^2 \rangle\text{ for all }c\text{ satisfying \eqref{cond:ddc}}.
\end{equation}
For every stopping time $T\in[0,{\hat{T}}]$, we define $c:=(\mathbbm{1}_{\{\tilde\delta^2>0\}}+\lambda\mathbbm{1}_{\{\tilde\delta^2=0\}})\mathbbm{1}_{(T,{\hat{T}})}$. Since $c\geq\lambda\geq\lambda\bar{c}$ on $(T,{\hat{T}})$, it is easy to see that $c$ satisfies \eqref{cond:ddc}. Therefore, \eqref{eq:lambda-ord-aux} implies that for all stopping times $T\in[0,{\hat{T}}]$,
$$\mathbb{E}\left[\int_T^{\hat{T}} \tilde\delta^2d\kappa\right]=\langle c,\tilde\delta^2\rangle\leq \langle c,\delta^2\rangle=\mathbb{E}\left[\int_T^{\hat{T}} \lambda\delta^2 d\kappa\right],$$
where we have used that $\{\delta^2>0\}\subseteq\{\tilde\delta^2=0\}$. By \eqref{eq:chron-ord-characterization}, this means that $\tilde\delta^2\preceq\lambda\delta^2$.
\end{proof}

%%%%%%%%%%%%%%%%%%%%%%%%%%%%

\subsection{An important lemma}
\begin{lem}\label{lem:important}
For all $\lambda\in[0,1]$ and $c\geq 0$ optional, the following holds:
\begin{equation}\label{eq:important}
\langle c\vee\lambda\bar{c},\delta\rangle=\sup_{\tilde{\delta}\preceq_\lambda \delta}\langle c,\tilde{\delta}\rangle\quad\text{for all}\quad\delta\geq0\text{ optional}.
\end{equation}
\end{lem}
Informally, in order to achieve the value of  $\langle c,\tilde{\delta}\rangle$ close to the value on the left-hand side of \eqref{eq:important}, we obtain $\tilde\delta d\kappa$ by splitting the mass of $\delta d\kappa$ into two parts: (i) if $c_t\geq\lambda\bar{c}_t$ then we leave the mass of $\delta d\kappa$ at time $t$ where it is (i.e., the corresponding part of $\delta$ goes to $\delta^1$ of Definition~\ref{def:lambda-ordering}), (ii) if $c_t<\lambda\bar{c}_t$ then we move the mass of $\delta d\kappa$ at time $t$ (i.e., the corresponding part of $\delta$ goes to $\delta^2$ of Definition~\ref{def:lambda-ordering}) to an earlier time where $c$ is close to its current running essential supremum $\bar{c}_t$, with the penalty of $(1-\lambda)$ fraction of the mass moved. 
\begin{proof}
The inequality ``$\geq$'' holds since for every $\tilde{\delta}\preceq_\lambda\delta$,
$$\langle c,\tilde\delta \rangle \leq \langle c\vee\lambda\bar{c},\tilde\delta\rangle\leq\langle c\vee\lambda\bar{c},\delta \rangle,$$
where the second inequality holds by \eqref{eq:monotonicity-lambda-ord}, since $c\vee\lambda\bar{c}$ satisfies \eqref{cond:ddc}. The opposite inequality ``$\leq$'' in \eqref{eq:important} will be proved in four steps.
\begin{enumerate}
\item[\textbf{Step 1.}] \textbf{$\mathbf{\lambda=1}$ and $\mathbf{c=\mathbbm{1}_A}$, where $\mathbf{A\in\mathcal{O}}$.} Since $c\vee\bar{c}=\bar{c}$, $\mathbb{P}\times d\kappa-$a.e., by Proposition~\ref{prop:increasing-dom}, we are proving that
$\langle\bar{c},\delta \rangle\leq\sup_{\tilde{\delta}\preceq\delta}\langle c,\tilde{\delta}\rangle$ for every optional process $\delta\geq 0$.
Let $\tau_A$ be the essential debut of $A$. Then $\tau_A$ is a stopping time by Proposition~\ref{prop:stopping-time} applied with $l=1$, and $\bar{c}_t=\mathbbm{1}_{(\tau_A,\hat{T})}(t)$ by the definition of $\bar{c}$.

Let us fix a non-negative optional process $\delta$. For an arbitrary $\varepsilon>0$ and for $\tau_A^\varepsilon:=(\tau_A+\varepsilon)\wedge\hat{T}$, the $\omega$-section of the set $(\tau_A,\tau_A^\varepsilon)\cap A$ has a positive Borel measure, and hence a positive measure $d\kappa$, as long as $\tau_A(\omega)<\hat{T}$. Therefore, a non-negative optional process $\delta^\varepsilon$ defined as the optional projection of the non-negative jointly measurable process
\begin{equation}\label{eq:delta-eps-def}
\begin{cases}
\frac{\mathbbm{1}_{(\tau_A,\tau_A^\varepsilon)\cap A}}{d\kappa((\tau_A,\tau_A^\varepsilon)\cap A)}\int_{\tau_A^\varepsilon}^{\hat{T}}\delta d\kappa,\quad\text{if }\tau_A<\hat{T},\\
0,\quad\text{otherwise},
\end{cases}
\end{equation}
is well-defined. Note that
\begin{equation*}
\mathbb{E}\left[\int_0^{\hat{T}} \mathbbm{1}_A\delta^\varepsilon d\kappa\right]=\mathbb{E}\left[\int_0^{\hat{T}} \mathbbm{1}_A\frac{\mathbbm{1}_{(\tau_A,\tau_A^\varepsilon)\cap A}}{d\kappa((\tau_A,\tau_A^\varepsilon)\cap A)}\left(\int_{\tau_A^\varepsilon}^{\hat{T}}\delta d\kappa\right) d\kappa\right]=\mathbb{E}\left[\int_{\tau_A^\varepsilon}^{\hat{T}}\delta d\kappa\right].
\end{equation*}
As $\varepsilon\downarrow 0$,
\begin{equation}\label{eq:eps-conv}
\langle c,\delta^\varepsilon \rangle=\mathbb{E}\left[\int_{\tau_A^\varepsilon}^{\hat{T}}\delta d\kappa\right]\nearrow\mathbb{E}\left[\int_{\tau_A}^{\hat{T}}\delta d\kappa\right]=\langle \bar{c},\delta \rangle,
\end{equation}
which proves the desired inequality as soon as we show that $\delta^\varepsilon\preceq\delta$ for all $\varepsilon>0$. As Remark \ref{rem:chron-ord-def}(i) says, it is sufficient to prove that for every stopping time $T$,
\begin{equation}\label{eq:delta-epsilon2}
\mathbb{E}\left[\int_T^{\hat{T}}\delta^\varepsilon d\kappa\right]\leq \mathbb{E}\left[\int_T^{\hat{T}}\delta d\kappa\right].
\end{equation}
Finally, \eqref{eq:delta-epsilon2} follows from
\begin{equation*}
\begin{aligned}
\mathbb{E}\left[\int_T^{\hat{T}}\delta^\varepsilon d\kappa\right]&=\mathbb{E}\left[\int_0^{\hat{T}}\mathbbm{1}_{[T,{\hat{T}})}\frac{\mathbbm{1}_{(\tau_A,\tau_A^\varepsilon)\cap A}}{d\kappa((\tau_A,\tau_A^\varepsilon)\cap A)}\left(\int_{\tau_A^\varepsilon}^{\hat{T}}\delta d\kappa\right)d\kappa\right]\\
&=\mathbb{E}\left[\frac{d\kappa((\tau_A,\tau_A^\varepsilon)\cap A\cap[T,{\hat{T}}))}{d\kappa((\tau_A,\tau_A^\varepsilon)\cap A)}\left(\int_{\tau_A^\varepsilon}^{\hat{T}}\delta d\kappa\right)\right]\\
&\leq \mathbb{E}\left[0\cdot \mathbbm{1}_{\{T\geq\tau_A^\varepsilon\}}\right]+\mathbb{E}\left[\left(\int_{\tau_A^\varepsilon}^{\hat{T}}\delta d\kappa\right)\mathbbm{1}_{\{T<\tau_A^\varepsilon\}}\right]\leq \mathbb{E}\left[\int_T^{\hat{T}}\delta d\kappa\right].\\
\end{aligned}
\end{equation*}
\item[\textbf{Step 2.}] \textbf{$\mathbf{\lambda\in[0,1]}$ and $\mathbf{c=\mathbbm{1}_A}$, where $\mathbf{A\in\mathcal{O}}$.} Let $\tau_A$ be the essential debut of $A$. Then $$\bar{c}=\mathbbm{1}_{(\tau_A,{\hat{T}})}\quad \text{and}\quad c\vee\lambda\bar{c}=\mathbbm{1}_{A}+\lambda\mathbbm{1}_{(\tau_A,{\hat{T}})}\mathbbm{1}_{A^c}.$$ We fix a non-negative optional process $\delta$, split it into $\delta^1=\delta\mathbbm{1}_{A}$ and $\delta^2=\delta\mathbbm{1}_{A^c}$, and define $\delta^\varepsilon=\tilde\delta^1+\tilde\delta^{2,\varepsilon}$ for $\varepsilon>0$ as follows: $\tilde\delta^1=\delta\mathbbm{1}_{A}$ and $\tilde\delta^{2,\varepsilon}$ is constructed as in Step~1 as the optional projection of \eqref{eq:delta-eps-def} for $c=\mathbbm{1}_{A}$ and $\delta$ replaced with $\lambda\delta\mathbbm{1}_{A^c}$. This ensures that (i) by \eqref{eq:delta-epsilon2}, $\tilde\delta^{2,\varepsilon}\preceq\lambda\delta\mathbbm{1}_{A^c}$, in particular, $\delta^\varepsilon=\tilde\delta^1+\tilde\delta^{2,\varepsilon}\preceq_\lambda \delta^1+\delta^2=\delta$, and (ii) by \eqref{eq:eps-conv},
\begin{equation*}
\langle c,\delta^\varepsilon \rangle=\langle c, \delta\mathbbm{1}_{A}+\tilde\delta^{2,\varepsilon} \rangle\nearrow\langle c,\delta\mathbbm{1}_{A}\rangle+\langle\bar{c},\lambda\delta\mathbbm{1}_{A^c} \rangle=\langle \mathbbm{1}_{A}+\mathbbm{1}_{(\tau_A,{\hat{T}})}\lambda\mathbbm{1}_{A^c},\delta \rangle=\langle c\vee\lambda\bar{c},\delta \rangle
\end{equation*}
as $\varepsilon\downarrow 0$, completing the proof of inequality ``$\leq$'' in \eqref{eq:important} for this case.

\item[\textbf{Step 3.}] \textbf{Sum of two optional processes with disjoint supports.} Let $\lambda\in[0,1]$ and let $c^1,c^2$ be two non-negative processes with disjoint supports for which \eqref{eq:important} holds. In this case, it follows directly from the definition of running essential supremum that $\overline{c^1+c^2}=\overline{(c^1\vee c^2)}=\bar{c}^1\vee\bar{c}^2$, and therefore
$$(c^1+c^2)\vee\lambda\overline{(c^1+c^2)}=(c^1\vee c^2)\vee\lambda(\bar{c}^1\vee\bar{c}^2)=(c^1\vee\lambda\bar{c}^1)\vee(c^2\vee\lambda\bar{c}^2).$$
We denote $c^{\lambda,i}:=c^i\vee\lambda\bar{c}^i$ for $i=1,2$ and obtain
\begin{equation*}
\begin{aligned}
\langle(c^1+c^2)\vee\lambda\overline{(c^1+c^2)},\delta\rangle&=\langle c^{\lambda,1}\vee c^{\lambda,2},\delta \rangle=\langle c^{\lambda,1},\delta\mathbbm{1}_{\{c^{\lambda,1}>c^{\lambda,2}\}} \rangle+\langle c^{\lambda,2},\delta\mathbbm{1}_{\{c^{\lambda,1}\leq c^{\lambda,2}\}} \rangle\\
&\overset{\eqref{eq:important}}{=}\sup\left\{\langle c^1,\tilde\delta^1 \rangle:\ \tilde\delta^1\preceq_\lambda \delta\mathbbm{1}_{\left\{c^{\lambda,1}>c^{\lambda,2}\right\}}\right\}\\
&\quad\quad+\sup\left\{\langle c^2,\tilde\delta^2 \rangle:\ \tilde\delta^2\preceq_\lambda \delta\mathbbm{1}_{\left\{c^{\lambda,1}\leq c^{\lambda,2}\right\}}\right\}\\
&\leq\sup_{\tilde\delta^1,\tilde\delta^2}\langle c^1+c^2,\tilde\delta^1+\tilde\delta^2\rangle \leq\sup_{\tilde\delta\preceq_\lambda\delta}\langle c^1+c^2,\tilde\delta \rangle,
\end{aligned}
\end{equation*}
where the last inequality holds by the additive property of $\preceq_\lambda$:
$$\tilde\delta^1\preceq_\lambda \delta\mathbbm{1}_{\left\{c^{\lambda,1}>c^{\lambda,2}\right\}}\text{ and }\tilde\delta^2\preceq_\lambda \delta\mathbbm{1}_{\left\{c^{\lambda,1}\leq c^{\lambda,2}\right\}}\quad\Rightarrow\quad \tilde\delta:=\tilde\delta^1+\tilde\delta^2\preceq_\lambda \delta.$$
Comparing the first and the last terms in the above sequence of equalities and inequalities shows that ``$\leq$'' in \eqref{eq:important} holds for $c^1+c^2$.
\item[\textbf{Step 4.}] \textbf{Monotone convergence.} Let $(c^n)_{n\geq1}$ be an increasing sequence of non-negative optional processes such that $c^n\uparrow c$ and
$$\langle c^n\vee\lambda\bar{c}^n,\delta \rangle\leq\sup_{\tilde{\delta}\preceq_\lambda \delta}\langle c^n,\tilde{\delta}\rangle\quad \text{for every } n.$$
Then since $\bar{c}^n\uparrow \bar{c}$, and therefore $c^n\vee\lambda\bar{c}^n\uparrow c\vee\lambda\bar{c}$, monotone convergence implies \eqref{eq:important}.
\end{enumerate}
Combining the results of Steps 2 and 3 shows that the conclusion of the lemma holds for all $\lambda\in[0,1]$ in case when $c$ is a non-negative simple $\mathcal{O}$-measurable function on $\Omega\times[0,\hat{T})$. Step~4 allows to conclude the proof for an arbitrary non-negative optional process $c$.
\end{proof}

%%%%%%%%%%%%%%%%%%%%%%%%%%%%

\subsection{The dual domain}

For $\lambda\in[0,1]$, we define the dual domain $\mathcal{D}^\lambda$ as the polar set of $\mathcal{C}^\lambda$ in $L_+^0(\Omega\times[0,\hat{T}),\mathcal{O},\mathbb{P}\times d\kappa)$ in the sense of \cite{bipolar}: $$\mathcal{D}^\lambda:=(\mathcal{C}^\lambda)^\circ=\left\{\delta\geq 0\text{ optional}: \langle c,\delta \rangle\leq 1\text{ for all }c\in\mathcal{C}^\lambda\right\},$$ 
and denote $\mathcal{D}^\lambda(y)=y\cdot\mathcal{D}^\lambda$ for $y>0$. By the bipolar theorem (\cite{bipolar}) and Proposition \ref{prop:C-bipolar}, we have $\mathcal{C}^\lambda=(\mathcal{D}^\lambda)^\circ$. As a polar set, $\mathcal{D}^\lambda$ is solid, convex, and closed. Furthermore, $\mathcal{D}^\lambda$ has the following characterization in terms of $\mathcal{Z}$:
\begin{prop}\label{prop:min-of-D}
The set $\mathcal{D}^\lambda$ is $\preceq_\lambda$-solid, i.e., if $\delta\in\mathcal{D}^\lambda$ and $\tilde{\delta}\preceq_\lambda\delta$ then $\tilde\delta\in\mathcal{D}^\lambda$. Moreover, it is the smallest set in $L_+^0(\Omega\times[0,\hat{T}),\mathcal{O},\mathbb{P}\times d\kappa)$ that is convex, closed, $\preceq_\lambda$-solid, and contains $\mathcal{Z}$.
\end{prop}
\begin{proof}
Let $\delta\in\mathcal{D}^\lambda$ and $\tilde{\delta}\preceq_\lambda\delta$. For every $c\in\mathcal{C}^\lambda$, $c\vee\lambda\bar{c}$ satisfies \eqref{cond:ddc} and belongs to $\mathcal{C}^\lambda$, and therefore
$$\langle c,\tilde{\delta} \rangle\leq\langle c\vee\lambda\bar{c},\tilde{\delta} \rangle\overset{\eqref{eq:monotonicity-lambda-ord}}{\leq}\langle c\vee\lambda\bar{c},\delta \rangle\leq 1,$$
implying $\tilde\delta\in(\mathcal{C}^\lambda)^\circ=\mathcal{D}^\lambda$, i.e., $\mathcal{D}^\lambda$ is $\preceq_\lambda$-solid. Let $\mathcal{D}'$ be the smallest set in $L_+^0(\Omega\times[0,\hat{T}),\mathcal{O},\mathbb{P}\times d\kappa)$ that is convex, closed, $\preceq_\lambda$-solid, and contains $\mathcal{Z}$. Clearly, $\mathcal{D}'\subseteq\mathcal{D}^\lambda$. We will show $\mathcal{D}^\lambda\subseteq\mathcal{D}'$ by proving that $(\mathcal{D}')^\circ\subseteq(\mathcal{D}^\lambda)^\circ=\mathcal{C}^\lambda$. Fix $c\in(\mathcal{D}')^\circ$. By Lemma \ref{lem:important} and since $\mathcal{D}'$ is $\preceq_\lambda$-solid and contains $\mathcal{Z}$, we obtain
$$\langle c\vee\lambda\bar{c},Z \rangle=\sup_{\tilde\delta\preceq_\lambda Z}\langle c,\tilde\delta \rangle\leq 1,\quad \forall Z\in\mathcal{Z},$$
which means that $c\vee\lambda\bar{c}$ is $1$-admissible and therefore $c\in\mathcal{C}^\lambda$.
\end{proof}
\begin{rem}[Monotonicity in $\lambda$]\label{rem:mon}
By the definition \eqref{def:primal-dom} of primal domains,
$$\mathcal{C}^1\subseteq\mathcal{C}^{\lambda_2}\subseteq\mathcal{C}^{\lambda_1}\subseteq\mathcal{C}^0\quad \text{for}\quad1\geq\lambda_2\geq\lambda_1\geq0,$$
where $\mathcal{C}^0$ is the set of all $x$-admissible consumption plans. This reflects the fact that the drawdown constraint gets more restrictive as $\lambda\in[0,1]$ increases.
The dual domains $\mathcal{D}^\lambda$, being the Brannath--Schachermayer duals of $\mathcal{C}^\lambda$, have the reverse monotonicity property:
$$\mathcal{D}^1\supseteq\mathcal{D}^{\lambda_2}\supseteq\mathcal{D}^{\lambda_1}\supseteq\mathcal{D}^0\quad \text{for}\quad1\geq\lambda_2\geq\lambda_1\geq0,$$
where, according to Proposition \ref{prop:min-of-D}, $\mathcal{D}^0$ is the smallest solid, convex, and closed subset of $L_+^0(\Omega\times[0,\hat{T}),\mathcal{O},\mathbb{P}\times d\kappa)$ containing $\mathcal{Z}$.
\end{rem}
\begin{prop}\label{prop:alpha-for-D}
For every $\lambda\in[0,1]$, the following holds:
\begin{equation*}
\alpha:=\sup_{Z\in\mathcal{Z}}\mathbb{E}\left[\int_0^{\hat{T}} Zd\kappa\right]=\sup_{\delta\in\mathcal{D}^\lambda}\mathbb{E}\left[\int_0^{\hat{T}}\delta d\kappa\right]\in(0,\infty).
\end{equation*}
\end{prop}
\begin{proof}
For a fixed $\lambda$, let us define
$$\mathcal{D}':=\left\{\delta\in\mathcal{D}^\lambda: \mathbb{E}\left[\int_0^{\hat{T}}\delta d\kappa\right]\leq \alpha\right\}.$$
Clearly, $\mathcal{Z}\subseteq\mathcal{D}'$. On the other hand, the set $\mathcal{D}'$ inherits the properties of $\mathcal{D}^\lambda$: it is convex, it is closed by Fatou's lemma, and it is $\preceq_\lambda$-solid, since it is easy to check that $\mathbb{E}\left[\int_0^{\hat{T}}\tilde\delta d\kappa\right]\leq\mathbb{E}\left[\int_0^{\hat{T}}\delta d\kappa\right]$ for $\tilde\delta\preceq_\lambda\delta$. By minimality of $\mathcal{D}^\lambda$ established in Proposition \ref{prop:min-of-D}, we must have $\mathcal{D}'=\mathcal{D}^\lambda$ and therefore $\sup_{\delta\in\mathcal{D}^\lambda}\mathbb{E}\left[\int_0^{\hat{T}} \delta d\kappa\right]=\alpha$.

Using Ito's formula for semimartingales and localization, one can show that $\mathbb{E}\left[\int_0^T Zd\kappa\right]=\mathbb{E}[Z_T\kappa_T]\leq A$ for every finite $T\leq\hat{T}$, $Z\in\mathcal{Z}$, and for constant $A$ from Assumption \ref{ass:clock}. Sending $T\to\infty$ if $\hat{T}=\infty$ and taking $T=\hat{T}$ otherwise, we conclude that $\alpha\leq A<\infty$.
\end{proof}
%%%%%%%%%%%%%%%%%%%%%%%%%%%%

\section{The optimization problem}\label{sec:optimization}

In this section, we describe the solution of the optimization problem \eqref{eq:primal-problem} using convex duality methods. We fix $\lambda\in(0,1]$ and for brevity omit mentioning the dependence on $\lambda$ explicitly, wherever it is possible. That is, we omit $\lambda$ from notations $\mathcal{C}^\lambda$, $\mathcal{D}^\lambda$, $\mathcal{C}^\lambda(x)$, $\mathcal{D}^\lambda(y)$, $\mathcal{C}^\lambda(x,q)$. The domain $\mathcal{K}$ for parameters $(x,q)$, the domain $\mathcal{L}$ for dual variables $(y,r)$, dual domains $\mathcal{D}(y,r)$, as well as the value function $u$ from \eqref{eq:primal-problem}, the dual value function $v$ defined in \eqref{eq:dual-problem}, and the primal and dual optimizers, \textit{all implicitly depend on $\lambda\in(0,1]$}.

Note that the definition \eqref{def:primal-dom} of the primal domains $\mathcal{C}(x,q)$ (i.e., $\mathcal{C}^\lambda(x,q)$) can be reformulated in terms of $\mathcal{C}(x)$ (i.e., $\mathcal{C}^\lambda(x)$) as follows:
\begin{equation}\label{eq:2nd-def-of-Cxq}
\mathcal{C}(x,q)=\{c\geq 0 \text{ optional}: c\vee \lambda q\in\mathcal{C}(x)\}\subseteq \mathcal{C}(x),
\end{equation}
without any reference to the drawdown constraint, only to the essential lower bound on consumption. In fact, all of the statements in this section (except Propositions \ref{prop:D-y-r-solid} and \ref{prop:sufficient-for-Dyr2}) hold for arbitrary subsets $\mathcal{C}$, $\mathcal{D}$, and $\mathcal{C}(x,q)$ defined as in \eqref{eq:2nd-def-of-Cxq}, of  $L_+^0(\Omega\times[0,\hat{T}),\mathcal{O},\mathbb{P}\times d\kappa)$ satisfying the following three conditions:
\begin{enumerate}
\item The sets $\mathcal{C}$ and $\mathcal{D}$ are polar to each other in the sense of \cite{bipolar}.
\item The set $\mathcal{D}$ contains an element that is strictly positive $\mathbb{P}\times d\kappa-$almost everywhere.
\item $\alpha:=\sup_{\delta\in\mathcal{D}}\mathbb{E}\left[\int_0^{\hat{T}} \delta d\kappa\right]<\infty$.
\end{enumerate}
All of these conditions hold for $\mathcal{C}^\lambda$ and $\mathcal{D}^\lambda$, $\lambda\in(0,1]$, the second one being satisfied since $\mathcal{Z}\subseteq\mathcal{D}^\lambda$. Additionally, this means that the arguments of this section can be adopted in order to introduce, with the definition analogous to \eqref{eq:2nd-def-of-Cxq}, essential lower bound on consumption in the usual unconstrained model of \cite{mostovyi}.
 
 %%%%%%%%%%%%%%%%%%%%%%%
 %%%%%%%%%%%%%%%%%%%%%%%

\subsection{Dual relations between domains}\label{sec:dual-rel}
Due to Proposition \ref{prop:alpha-for-D}, the set $\mathcal{C}(x,q)$ for $x>0$ is non-empty if and only if $\lambda q\leq x/\alpha$.
This leads to the definition of two cones in $\mathbb{R}^2$,
\begin{equation*}
\begin{aligned}
\bar{\mathcal{K}}&=\left\{(x,q): x\geq 0 \text{ and } x/(\alpha\lambda)\geq  q\right\},\\
\bar{\mathcal{L}}&=\left\{(y,r): xy+qr\geq 0\text{ for all }(x,q)\in\bar{\mathcal{K}}\right\}=\left\{(y,r):y\geq 0 \text{ and }  0\geq r\geq -(\alpha\lambda) y\right\},
\end{aligned}
\end{equation*}
and their interiors $\mathcal{K}=\text{int}(\bar{\mathcal{K}})$, $\mathcal{L}=\text{int}(\bar{\mathcal{L}})$. The open cone $\mathcal{K}$ is precisely the set of pairs $(x,q)$ such that the optimization problem \eqref{eq:primal-problem} is non-trivial, the closed cone $\bar{\mathcal{L}}$ is the negative of the polar cone of $\mathcal{K}$ in $\mathbb{R}^2$. Further, we define the set
$$\mathcal{L}^*:=\mathcal{L}\cup\{(y,r):y>0,r=0\}\subseteq\bar{\mathcal{L}}$$
that turns out to be the appropriate domain for the dual value function. The dual domains are defined as follows (cf. (11) in \cite{hug-kramkov} and (4.4) in \cite{yu}):
\begin{equation*}
\mathcal{D}(y,r):=\left\{\delta\in\mathcal{D}(y):\langle c,\delta \rangle\leq xy+qr,\ \forall \ (x,q)\in\mathcal{K},\ c\in\mathcal{C}(x,q)\right\}, \quad (y,r)\in\mathcal{L}^*.
\end{equation*}
Some trivial but useful properties of the domains $(\mathcal{C}(x,q))_{(x,q)\in\mathcal{K}}$ and $(\mathcal{D}(y,r))_{(y,r)\in\mathcal{L}^*}$ are summarized in the following remark.
\begin{rem}
\begin{enumerate}
\item[(i)] We have $\mathcal{C}(x,q)=\mathcal{C}(x)$ for $q\leq 0$. For all $(x,q)\in\mathcal{K}$, the set $\mathcal{C}(x,q)$ is solid and contains the constant process $c\equiv x/\alpha>0$. For $x>0$ and $q_1<q_2<x/(\alpha\lambda)$, we have $\mathcal{C}(x,q_2)\subseteq\mathcal{C}(x,q_1)$, i.e., the sets $\mathcal{C}(x,q)$ decrease as $q$ increases.
\item[(ii)] We have $\mathcal{D}(y,0)=\mathcal{D}(y)$ for all $y>0$.
Since $\mathcal{C}(x,q)=\mathcal{C}(x)$ for $q\leq 0$ and since $r\leq0$ for every $(y,r)\in\mathcal{L}^*$, the definition of $\mathcal{D}(y,r)$ is equivalent to
\begin{equation}\label{eq:def-d-spaces}
\mathcal{D}(y,r)=\{\delta\in\mathcal{D}(y):\langle c,\delta \rangle\leq xy+qr,\ \forall\ q\geq 0, \ (x,q)\in\mathcal{K}, \ c\in\mathcal{C}(x,q)\}.
\end{equation}
In particular, $\mathcal{D}(y,r_1)\subseteq\mathcal{D}(y,r_2)$ for all $y>0$ and $-(\alpha\lambda) y<r_1<r_2\leq 0$, i.e., the sets $\mathcal{D}(y,r)$ increase as $r$ increases.
\end{enumerate}
\end{rem}
\begin{prop}\label{prop:D-y-r-solid}
The set $\mathcal{D}(y,r)$ is $\preceq_\lambda$-solid for every $(y,r)\in\mathcal{L}^*$.
\end{prop}
\begin{proof}
Assume $\delta\in\mathcal{D}(y,r)$ and $\tilde\delta\preceq_\lambda\delta$. Let $(x,q)\in\mathcal{K}$ and $c\in\mathcal{C}(x,q)$. Then, by Proposition \ref{prop:C-is-closed}, $c\vee\lambda\bar{c}\in\mathcal{C}(x,q)$ and satisfies \eqref{cond:ddc}, therefore
$$\langle c,\tilde{\delta}\rangle\leq\langle c\vee\lambda\bar{c},\tilde{\delta}\rangle\overset{\eqref{eq:monotonicity-lambda-ord}}{\leq}\langle c\vee\lambda\bar{c},\delta \rangle\leq xy+rq,$$
implying that $\tilde\delta\in\mathcal{D}(y,r)$, i.e., $\mathcal{D}(y,r)$ is $\preceq_\lambda$-solid.
\end{proof}

%%%%%%%%%%%%%%
The following is a simple sufficient condition for an element of $\mathcal{D}(y)$ in order to be an element of $\mathcal{D}(y,r)$. It will be used in Section \ref{sec:complete} to characterize the optimizers for \eqref{eq:primal-problem} in a complete market.

\begin{prop}\label{prop:sufficient-for-Dyr2}
Let $\tilde\delta\preceq_\lambda\delta\in\mathcal{D}(y)$ and $\tilde\delta$ is strictly positive on a set of positive measure (i.e., excluding the trivial case $\tilde\delta=0$, $\mathbb{P}\times d\kappa-$a.e.). Then $\tilde\delta\in\mathcal{D}(y,r)$ for
\begin{equation*}
r:=\mathbb{E}\left[\int_0^{\hat{T}} [(\tilde\delta-\delta)\vee 0-\lambda(\delta-\tilde\delta)\vee0] d\kappa\right].
\end{equation*}
\end{prop}
\begin{proof}
First, since $(\tilde\delta-\delta)\vee 0\preceq\lambda(\delta-\tilde\delta)\vee0$ by Proposition \ref{prop:lambda-ord-alternative}, $r\leq 0$. Second,
$$r\geq -\mathbb{E}\left[\int_0^{\hat{T}} [\lambda(\delta-\tilde\delta)\vee0] d\kappa\right]\geq -\mathbb{E}\left[\int_0^{\hat{T}} \lambda\delta d\kappa\right]\geq -(\alpha\lambda)y,$$
and the strict inequality holds between the first and the last terms because $\tilde\delta>0$ on a set of positive measure. Hence, $(y,r)\in\mathcal{L}^*$ so that $\mathcal{D}(y,r)$ is well-defined. According to \eqref{eq:def-d-spaces}, it is enough to test the defining property of $\mathcal{D}(y,r)$ only for $(x,q)\in\mathcal{K}$, $c\in\mathcal{C}(x,q)$ with $q\geq 0$.
Since $\bar{c}\vee q\geq c\geq \lambda(\bar{c}\vee q)$, $\mathbb{P}\times d\kappa-$a.e.,
\begin{equation*}
\begin{aligned}
\langle c,\tilde\delta-\delta\rangle&=\langle c,(\tilde\delta-\delta)\vee 0\rangle-\langle c,(\delta-\tilde\delta)\vee 0\rangle\leq \langle \bar{c}\vee q,(\tilde\delta-\delta)\vee 0\rangle-\langle \bar{c}\vee q,\lambda(\delta-\tilde\delta)\vee 0\rangle\\
&=\langle \bar{c}\vee q,(\tilde\delta-\delta)\vee 0-\lambda(\delta-\tilde\delta)\vee0\rangle\leq \langle q, (\tilde\delta-\delta)\vee 0-\lambda(\delta-\tilde\delta)\vee0\rangle=qr,
\end{aligned}
\end{equation*}
where the last inequality follows from $(\bar{c}\vee q)-q\in\mathcal{C}_{\text{inc}}$ and $(\tilde\delta-\delta)\vee 0\preceq\lambda(\delta-\tilde\delta)\vee0$. Rearranging the terms and using $c\in\mathcal{C}(x)$, $\delta\in\mathcal{D}(y)$, we obtain $\langle c,\tilde\delta\rangle\leq\langle c,\delta\rangle+qr\leq xy+qr$, completing the proof.
\end{proof}
%%%%%%%%%%%%%%

Now we prove duality relations between the families $(\mathcal{C}(x,q))_{(x,q)\in\mathcal{K}}$ and $(\mathcal{D}(y,r))_{(y,r)\in\mathcal{L}^*}$. The next proposition is an analogue of Proposition 1 in \cite{hug-kramkov} for our setting and is crucial for reducing, similarly to \cite{hug-kramkov}, the two-dimensional optimization problem to one dimension in the proof of Theorem \ref{thm:main-duality} below.
\begin{prop}\label{prop:conjugate-rel}
\begin{enumerate}
\item[(i)] For every $(x,q)\in\mathcal{K}$, $\mathcal{C}(x,q)$ contains a stictly positive process. For a non-negative optional process $c$, the following holds:
\begin{equation}\label{eq:C-characterization}
c\in\mathcal{C}(x,q)\quad\Leftrightarrow\quad \langle c,\delta \rangle\leq xy+qr\quad \text{for all } \ (y,r)\in\mathcal{L}^*,\ \delta\in\mathcal{D}(y,r).
\end{equation}
\item[(ii)] For every $(y,r)\in\mathcal{L}^*$, $\mathcal{D}(y,r)$ contains a strictly positive process. For a non-negative optional process $\delta$, the following holds:
\begin{equation}\label{eq:D-characterization}
\delta\in\mathcal{D}(y,r)\quad\Leftrightarrow\quad \langle c,\delta \rangle\leq xy+qr\quad \text{for all } \ (x,q)\in\mathcal{K},\ c\in\mathcal{C}(x,q).
\end{equation}
\end{enumerate}
\end{prop}
\begin{proof}
We show (ii) first. The implication ``$\Rightarrow$'' in \eqref{eq:D-characterization} holds by the definition of $\mathcal{D}(y,r)$. Assuming the right-hand side of \eqref{eq:D-characterization} and substituting $q=0$ implies that $\delta\in\mathcal{D}(y)$, therefore the implication ``$\Leftarrow$'' also follows from the definition of $\mathcal{D}(y,r)$. To show that $\mathcal{D}(y,r)$ contains a stictly positive process, we will find an $\varepsilon>0$ such that $\mathcal{D}(\varepsilon)\subseteq\mathcal{D}(y,r)$. Since $\mathcal{Z}\subseteq\mathcal{D}$, the claim then follows from strict positivity of equivalent martingale deflators $Z\in\mathcal{Z}$. A sufficient condition for $\mathcal{D}(\varepsilon)\subseteq\mathcal{D}(y,r)$ is
$$\varepsilon x\leq xy+qr,\quad \forall (x,q)\in\mathcal{K},$$
since then for every $\delta\in\mathcal{D}(\varepsilon)$ and $c\in\mathcal{C}(x,q)\subseteq\mathcal{C}(x)$, $\langle c,\delta \rangle\leq \varepsilon x\leq xy+qr$. Equivalently, $\varepsilon\leq y+(q/x)r$ for all $(x,q)\in\mathcal{K}$. By the definition of the open cone $\mathcal{K}$ this means that $\varepsilon\leq y+r/(\alpha\lambda)$. Finally, by the definition of $\mathcal{L}^*$, $y+r/(\alpha\lambda)>0$, therefore such an $\varepsilon>0$ exists.

Now we show (i). The constant consumption plan $c\equiv x/\alpha$ belongs to $\mathcal{C}(x,q)$ and is strictly positive. The implication ``$\Rightarrow$'' in \eqref{eq:C-characterization} follows from the definition of $\mathcal{D}(y,r)$. To show ``$\Leftarrow$'' when $q\leq 0$, we conclude that $c\in\mathcal{C}(x)=\mathcal{C}(x,q)$ by testing the right-hand side of \eqref{eq:C-characterization} with $\delta\in\mathcal{D}=\mathcal{D}(1,0)$. It remains to prove ``$\Leftarrow$'' for the case $q>0$. Let us fix $\delta\in\mathcal{D}$. According to \eqref{eq:2nd-def-of-Cxq}, we want to show that the right-hand side of \eqref{eq:C-characterization} implies $\langle c\vee\lambda q,\delta \rangle\leq x$ or, equivalently,
\begin{equation}\label{eq:belong-to-Cx}
\langle\lambda q, \delta \rangle+\langle (c-\lambda q)\vee 0,\delta \rangle\leq x.
\end{equation}
The idea of the proof of this inequality below is to show that $\delta\mathbbm{1}_{\{c\geq\lambda q\}}\in\mathcal{D}(1,r)$ for an appropriate $r$ and then apply the right-hand side of \eqref{eq:C-characterization} to this process.

Denote $s:=\mathbb{E}\left[\int_0^{\hat{T}} \delta d\kappa\right]\in[0,\alpha]$. For every $q'\geq 0, (x',q')\in\mathcal{K}$, and $c'\in\mathcal{C}(x',q')$, we have $c'\vee\lambda q'\in\mathcal{C}(x')$. Therefore,
\begin{equation}\label{eq:test-delta}
\langle (c'-\lambda q')\vee0, \delta\mathbbm{1}_{\{c\geq\lambda q\}}\rangle \leq \langle (c'-\lambda q')\vee 0,\delta \rangle\leq x'-\lambda q's.
\end{equation}
Denote $p:=\mathbb{E}\left[\int_0^{\hat{T}} \delta\mathbbm{1}_{\{c\geq\lambda q\}} d\kappa\right]\in[0,s]$. If $p=0$, then $\delta=0$ almost everywhere on $\{c\geq\lambda q\}$ and \eqref{eq:belong-to-Cx} is proved, since $\lambda qs\leq \lambda q\alpha<x$.
Hence, we can assume that $p\in(0,s]$. By adding $\lambda q'p$ to the inequality \eqref{eq:test-delta}, we obtain
\begin{equation}\label{eq:belong-to-D}
\langle c'\vee\lambda q', \delta\mathbbm{1}_{\{c\geq\lambda q\}}\rangle \leq x'+\lambda q'(p-s),
\end{equation}
for all  $q'\geq 0$, $(x',q')\in\mathcal{K}$, and $c'\in\mathcal{C}(x',q')$. Since $0\geq (p-s)\lambda>-\alpha\lambda$, $(1,(p-s)\lambda)\in\mathcal{L}^*$ and the inequality \eqref{eq:belong-to-D} means that $\delta\mathbbm{1}_{\{c\geq\lambda q\}}\in\mathcal{D}(1,(p-s)\lambda)$. Therefore, we can test the right-hand side of \eqref{eq:C-characterization} with this process, which yields
$$\langle c, \delta\mathbbm{1}_{\{c\geq\lambda q\}}\rangle\leq x+\lambda q(p-s).$$
By subtracting $\lambda qp=\langle\lambda q,\delta\mathbbm{1}_{\{c\geq\lambda q\}} \rangle$ from both sides, we get
$$\langle (c-\lambda q)\vee 0,\delta \rangle=\langle c-\lambda q,\delta\mathbbm{1}_{\{c\geq\lambda q\}} \rangle\leq x-\lambda qs.$$
Moving $\lambda qs=\langle\lambda q,\delta \rangle$ to the left yields \eqref{eq:belong-to-Cx}, concluding the proof.
\end{proof}

%%%%%%%%%%%%%%%%%%%%%%%%%%%%

\subsection{The main duality result for optimization}\label{sec:duality-sol}

Following \cite{mostovyi} (and the standard duality theory), we introduce the conjugate stochastic field $V$ to $U$:
$$V(\omega,t,y):=\sup_{x>0}(U(\omega,t,x)-xy),\quad (\omega,t,y)\in\Omega\times[0,\hat{T})\times[0,\infty),$$
and denote by $I(\omega,t,y):=-V'(\omega,t,y)$, the derivative of $-V$ with respect to $y$. It is well-known that $-V$ satisfies Assumption \ref{ass:utility} and $I(\omega,t,\cdot)=(U')^{-1}(\omega,t,\cdot)$.

We define the dual optimization problem for \eqref{eq:primal-problem} through the value function
\begin{equation}\label{eq:dual-problem}
v(y,r):=\inf_{\delta\in\mathcal{D}(y,r)}\mathbb{E}\left[\int_0^{\hat{T}} V(\omega,t,\delta_t)d\kappa_t\right],\quad(y,r)\in\mathcal{L}^*,
\end{equation}
where the convention
$$\mathbb{E}\left[\int_0^{\hat{T}} V(\omega,t,\delta_t)d\kappa_t\right]:=+\infty\quad\text{if}\quad\mathbb{E}\left[\int_0^{\hat{T}} V^+(\omega,t,\delta_t)d\kappa_t\right]=+\infty$$
is used and $W^{+}$ denotes the positive part of a stochastic field $W$. We often omit writing down the dependence of $U$, $U'$, $V$, and $I$ on $\omega$ and $t$ in what follows.

The following theorem describes the duality between \eqref{eq:primal-problem} and \eqref{eq:dual-problem}, and establishes existence and uniqueness of optimizers under Mostovyi's finiteness assumptions for value functions $u$ and $v$.
\begin{thm}\label{thm:main-duality}
Suppose that $\lambda\in(0,1]$, Assumptions \ref{ass:clock}, \ref{ass:NA}, \ref{ass:utility} hold, and
$$u(x,q)>-\infty\text{  for all  }(x,q)\in\mathcal{K}\quad \text{and}\quad v(y,r)<\infty\text{  for all  }(y,r)\in\mathcal{L}^*.$$
Then we have
\begin{enumerate}
\item $u(x,q)<\infty$ for all $(x,q)\in\mathcal{K}$, $v(y,r)>-\infty$ for all $(y,r)\in\mathcal{L}^*$. The functions $u$ and $v$ are conjugate:
\begin{equation}\label{eq:conjugacy-rel}
\begin{aligned}
u(x,q)&=\inf_{(y,r)\in\mathcal{L}^*}\{v(y,r)+xy+qr\},&\quad (x,q)\in\mathcal{K},\\
v(y,r)&=\sup_{(x,q)\in\mathcal{K}}\{u(x,q)-xy-qr\},&\quad (y,r)\in\mathcal{L}^*.\\
\end{aligned}
\end{equation}
\item The subdifferential of $u$ maps $\mathcal{K}$ into $\mathcal{L}^*$:
$$\partial u(x,q)\subseteq\mathcal{L}^*\quad \text{for all}\quad(x,q)\in\mathcal{K}.$$
\item For all $(x,q)\in\mathcal{K}$ and $(y,r)\in\mathcal{L}^*$ the optimal solutions $\hat{c}(x,q)$ to \eqref{eq:primal-problem} and $\hat{\delta}(y,r)$ to \eqref{eq:dual-problem} exist and are unique. Moreover, if $(y,r)\in\partial u(x,q)$ then we have the dual relations
\begin{equation}\label{eq:dual-relations-thm}
\begin{aligned}
\hat{\delta}_t(y,r)&=U'(t,\hat{c}_t(x,q)),\quad \mathbb{P}\times d\kappa-\text{a.e.},\\
\langle\hat{c}(x,q),\hat{\delta}(y,r)\rangle&=xy+qr.
\end{aligned}
\end{equation}
\end{enumerate}
\end{thm}
\begin{rem}\label{rem:duality-thm}
\begin{enumerate}
\item[(i)]Since each of the sets $\mathcal{C}(x,q)$, $(x,q)\in\mathcal{K}$, contains the strictly positive constant consumption plan $c\equiv x/\alpha$, the condition that $u(x,q)>-\infty$ for all $(x,q)$ in $\mathcal{K}$ is satisfied if $\mathbb{E}\left[\int_0^{\hat{T}} U^-(t,z)d\kappa_t\right]<\infty$ for all $z>0$.
\item[(ii)] It follows from the proof of Proposition \ref{prop:conjugate-rel}(ii) that for all $(y,r)\in\mathcal{L}^*$ the set $\mathcal{D}(y,r)$ contains $\mathcal{D}(\varepsilon)$ for $\varepsilon>0$ small enough. Hence, the condition that $v(y,r)<\infty$ for all  $(y,r)\in\mathcal{L}^*$ is equivalent to $v(y,0)<\infty$ for all $y>0$.

Furthermore, since ${D}^0(\varepsilon)\subseteq\mathcal{D}(\varepsilon)$ by Remark \ref{rem:mon}, we have
$$v(y,r)\leq \inf_{\delta\in{D}^0(\varepsilon)}\mathbb{E}\left[\int_0^{\hat{T}} V(t,\delta_t)d\kappa_t\right]\quad \text{for some }\varepsilon>0.$$
The infimum on the right-hand side is the dual value function of the unconstrained problem, when both the drawdown and the essential lower bound constraints are dropped ($v(\varepsilon)$, in the notation of \cite{mostovyi}). Therefore, the condition that $v(y,r)<\infty$ for all $(y,r)\in\mathcal{L}^*$ is satisfied as long as the dual value function of the unconstrained problem is finite (from above), in particular, it is always true if the unconstrained problem satisfies the assumptions of \cite{mostovyi}, Theorem 2.3.
\item[(iii)]
Under Assumptions \ref{ass:clock}, \ref{ass:NA}, \ref{ass:utility} and assuming $u(x,0)>-\infty$ and $v(y,0)<\infty$ for all $x,y>0$ (in particular, under assumptions of Theorem \ref{thm:main-duality}) we can apply \cite[Theorem 3.2]{mostovyi} to the sets $\mathcal{C}:=\mathcal{C}^\lambda$ and $\mathcal{D}:=\mathcal{D}^\lambda$ for $\lambda\in(0,1]$ and thus solve the initial problem \eqref{eq:primal-problem} in case $q=0$. We use this observation in the proof of Proposition \ref{prop:complete-case} below. In particular, for the case $\lambda=1$, Theorem 3.2 in \cite{mostovyi} applied to the sets $\mathcal{C}^1$ and $\mathcal{D}^1$ is an analogue of Theorem 3.2 in \cite{BK}.
\end{enumerate}
\end{rem}
The following two lemmas are analogues of Lemmas 11 and 12 in \cite{hug-kramkov} for our setting and will help us prove Theorem \ref{thm:main-duality}.
\begin{lem}\label{lem:sup-over-set}
Let $\mathcal{E}\subseteq L_+^0(\Omega\times[0,{\hat{T}}),\mathcal{O},\mathbb{P}\times d\kappa)$ be a convex set. If for every $\varepsilon>0$ there exists $c^\varepsilon\in \varepsilon\mathcal{E}$ such that
$$\mathbb{E}\left[\int_0^{\hat{T}} U(t,c^\varepsilon_t)d\kappa_t\right]>-\infty,$$
then, for every $x>0$,
$$\sup_{c\in x\mathcal{E}}\mathbb{E}\left[\int_0^{\hat{T}} U(t,c_t)d\kappa_t\right]=\sup_{c\in x\text{cl}(\mathcal{E})}\mathbb{E}\left[\int_0^{\hat{T}} U(t,c_t)d\kappa_t\right],$$
where $\text{cl}(\mathcal{E})$ denotes the closure of $\mathcal{E}$ with respect to convergence in measure $\mathbb{P}\times d\kappa$.
\end{lem}
The proof below follows very closely the proof of Lemma 11 in \cite{hug-kramkov}, except a small adjustment due to the stochastic utility and working on the product space $\Omega\times[0,{\hat{T}})$.
\begin{proof}
Denote, for $x>0$,
$$\phi(x):=\sup_{c\in x\mathcal{E}}\mathbb{E}\left[\int_0^{\hat{T}} U(t,c_t)d\kappa_t\right]\quad\text{and}\quad \psi(x):=\sup_{c\in x\text{cl}(\mathcal{E})}\mathbb{E}\left[\int_0^{\hat{T}} U(t,c_t)d\kappa_t\right].$$
Clearly, $\phi$ and $\psi$ are concave functions and $\psi\geq\phi>-\infty$ on $(0,\infty)$. If $\phi(x)=\infty$ for some $x>0$, then, due to concavity, $\phi$ is infinite on entire $(0,\infty)$ and the assertion of the lemma is trivial. Hereafter we assume that $\phi$ is finite.

Fix $x>0$ and $c\in x\text{cl}(\mathcal{E})$. Let $(c^n)_{n\geq 1}$ be a sequence in $x\mathcal{E}$ that converges $\mathbb{P}\times d\kappa-$a.e. to $c$. For any $\varepsilon>0$, we have
\begin{align*}
\mathbb{E}\left[\int_0^{\hat{T}} U(t,c_t)d\kappa_t\right]&\leq \mathbb{E}\left[\int_0^{\hat{T}} U(t,c_t+c^\varepsilon_t)d\kappa_t\right]\\
&\leq \liminf_{n\to\infty}\mathbb{E}\left[\int_0^{\hat{T}} U(t,c^n_t+c^\varepsilon_t)d\kappa_t\right]\leq\phi(x+\varepsilon),
\end{align*}
where the first inequality holds true because $U$ is increasing, the second one follows from Fatou's lemma, since $U(t,c^n_t+c^\varepsilon_t)\to U(t,c_t+c^\varepsilon_t)$ almost surely and all the terms are bounded from below by the integrable process $-U^{-}(t, c^\varepsilon_t)$, and the third one follows from the fact that $\mathcal{E}$ is convex and therefore $c^n+c^\varepsilon\in (x+\varepsilon)\mathcal{E}$ for $n\geq 1$. Since $\phi$ is concave, it is continuous. It follows that
$$\psi(x)=\sup_{c\in x\text{cl}(\mathcal{E})}\mathbb{E}\left[\int_0^{\hat{T}} U(t,c_t)d\kappa_t\right]\leq\lim_{\varepsilon\downarrow 0}\phi(x+\varepsilon)=\phi(x).$$
\end{proof}

\begin{lem}\label{lem:limit-in-dual}
Let $(y_n,r_n)\in\mathcal{L}^*$ and $\delta^n\in\mathcal{D}(y_n,r_n)$, $n\geq 1$, converge to $(y,r)$ and to an optional process $\delta\geq 0$, respectively. If $\delta>0$, $\mathbb{P}\times d\kappa-$a.e., then $(y,r)\in\mathcal{L}^*$ and $\delta\in\mathcal{D}(y,r)$.
\end{lem}
\begin{proof}
Let $(x,q)\in\mathcal{K}$. Since the constant process $x/\alpha$ belongs to $\mathcal{C}(x,q)$, by Proposition~\ref{prop:conjugate-rel}, $\langle x/\alpha,\delta^n\rangle \leq xy_n+qr_n$ for $n\geq1$. Then, by Fatou's lemma,
\begin{equation}\label{eq:bdry-of-K}
0<\langle x/\alpha,\delta\rangle\leq xy+qr.
\end{equation}
Note that since the second inequality in \eqref{eq:bdry-of-K} holds for all $q<x/(\alpha\lambda)$, it holds for $q=x/(\alpha\lambda)$ by continuity as well. Since $xy+qr>0$ for all $(x,q)\in\mathcal{K}\cup\{(x',q'):x'>0, q'=x'/(\alpha\lambda)\}$, we have $(y,r)\in\mathcal{L}^*.$ Finally, Fatou's lemma and Proposition \ref{prop:conjugate-rel} imply that $\delta\in\mathcal{D}(y,r)$.
\end{proof}

\begin{proof}[Proof of Theorem \ref{thm:main-duality}]
Since $V(\omega,t,y)+xy\geq U(\omega,t,x)$ for all $x,y>0$ and $(\omega,t)\in\Omega\times[0,\hat{T})$ and since $v(y,r)<\infty$ for some $(y,r)\in\mathcal{L}^*$, we deduce from Proposition \ref{prop:conjugate-rel} that $u(x,q)<\infty$ for all $(x,q)\in\mathcal{K}$. Analogously, since $u(x,q)>-\infty$ for some $(x,q)\in\mathcal{K}$, $v(y,r)>-\infty$ for all $(y,r)\in\mathcal{L}^*$. Hence, $u$ and $v$ are both finite on $\mathcal{K}$ and $\mathcal{L}^*$, respectively.

For $(y,r)\in\mathcal{L}^*$, we define the sets
\begin{equation*}
\begin{aligned}
A(y,r)&=\{(x,q)\in\mathcal{K}: xy+qr\leq 1\},\\
\tilde{\mathcal{C}}^{(y,r)}&=\bigcup_{(x,q)\in A(y,r)}\mathcal{C}(x,q),
\end{aligned}
\end{equation*}
and denote by $\mathcal{C}^{(y,r)}$ the closure of $\tilde{\mathcal{C}}^{(y,r)}$ with respect to convergence in measure $\mathbb{P}\times d\kappa$. By Lemma \ref{lem:sup-over-set}, for $z>0$,
$$\sup_{c\in z\mathcal{C}^{(y,r)}}\mathbb{E}\left[\int_0^{\hat{T}} U(t,c_t)d\kappa_t\right]=\sup_{c\in z\tilde{\mathcal{C}}^{(y,r)}}\mathbb{E}\left[\int_0^{\hat{T}} U(t,c_t)d\kappa_t\right]=\sup_{(x,q)\in zA(y,r)} u(x,q)>-\infty.$$
By Proposition \ref{prop:conjugate-rel}, the sets $\mathcal{D}(y,r)$ and $\mathcal{C}^{(y,r)}$ are polar sets of each other. Therefore, they satisfy the assumptions of \cite{mostovyi}, Theorem 3.2. This theorem implies that there exists a unique solution $\hat{\delta}(y,r)$ to \eqref{eq:dual-problem} and the second conjugacy relation in \eqref{eq:conjugacy-rel} holds:
\begin{equation}\label{eq:1st-conj-rel}
\begin{aligned}
v(y,r)&=\sup_{z>0}\left\{\sup_{c\in z\mathcal{C}^{(y,r)}}\mathbb{E}\left[\int_0^{\hat{T}} U(t,c_t)d\kappa_t\right]-z\right\}=\sup_{z>0, (x,q)\in zA(y,r)}\{u(x,q)-z\}\\
&=\sup_{(x,q)\in\mathcal{K},z\geq xy+qr}\{u(x,q)-z\}=\sup_{(x,q)\in\mathcal{K}}\{u(x,q)-xy-qr\}.
\end{aligned}
\end{equation}
The function $u$ is clearly concave on $\mathcal{K}$. The first equation in \eqref{eq:conjugacy-rel} follows from \cite{rockafellar}, Section 12.

For $(x,q)\in\mathcal{K}$, we define the sets
\begin{equation*}
\begin{aligned}
B(x,q)&=\{(y,r)\in\mathcal{L}^*: xy+qr\leq 1\},\\
\tilde{\mathcal{D}}^{(x,q)}&=\bigcup_{(y,r)\in B(x,q)}\mathcal{D}(y,r),
\end{aligned}
\end{equation*}
and denote by $\mathcal{D}^{(x,q)}$ the closure of $\tilde{\mathcal{D}}^{(x,q)}$ with respect to convergence in measure $\mathbb{P}\times d\kappa$. Clearly, for $z>0$ we have
$$\inf_{\delta\in z\mathcal{D}^{(x,q)}}\mathbb{E}\left[\int_0^{\hat{T}} V(t,\delta_t)d\kappa_t\right]\leq \inf_{\delta\in z\tilde{\mathcal{D}}^{(x,q)}}\mathbb{E}\left[\int_0^{\hat{T}} V(t,\delta_t)d\kappa_t\right]=\inf_{(y,r)\in zB(x,q)}v(y,r)<\infty.$$
By Proposition \ref{prop:conjugate-rel}, the sets $\mathcal{C}(x,q)$ and $\mathcal{D}^{(x,q)}$ are polar sets of each other. Therefore, they satisfy the assumptions of \cite{mostovyi}, Theorem 3.2, and this theorem implies that there exists a unique solution $\hat{c}(x,q)$ to \eqref{eq:primal-problem}. Moreover, denoting
$$\hat{\delta}_t^{(x,q)}:=U'(t,\hat{c}_t(x,q))\quad\text{and}\quad z:=\langle \hat{c}(x,q),\hat{\delta}^{(x,q)}\rangle,$$
we deduce from \cite{mostovyi}, Theorem 3.2, that $\hat{\delta}_t^{(x,q)}\in z\mathcal{D}^{(x,q)}$ and that it is the unique solution of the optimization problem
$$\mathbb{E}\left[\int_0^{\hat{T}} V(t,\hat{\delta}_t^{(x,q)})d\kappa_t\right]=\inf_{\delta\in z\mathcal{D}^{(x,q)}}\mathbb{E}\left[\int_0^{\hat{T}} V(t,\delta_t)d\kappa_t\right].$$
Since $\hat{\delta}_t^{(x,q)}>0$, $\mathbb{P}\times d\kappa-$a.e., and since the set $zB(x,q)\subseteq\mathbb{R}^2$ is bounded, Lemma \ref{lem:limit-in-dual} implies the existence of $(y,r)\in zB(x,q)$ such that $\hat{\delta}_t^{(x,q)}\in\mathcal{D}(y,r)$. Since $\hat{\delta}_t^{(x,q)}\in\mathcal{D}(y,r)$ is the optimizer on $z\mathcal{D}^{(x,q)}$,
$$xy+qr=z=\langle \hat{c}(x,q),\hat{\delta}^{(x,q)}\rangle\quad\text{and}\quad \hat{\delta}_t(y,r)=\hat{\delta}_t^{(x,q)}=U'(t,\hat{c}_t(x,q)).$$
Further, the pointwise equality $U(t,\hat{c}_t(x,q))=V(t,\hat{\delta}_t(y,r))+\hat{c}_t(x,q)\hat{\delta}_t(y,r)$ implies
\begin{equation}\label{eq:equality}
u(x,q)=v(y,r)+xy+qr,
\end{equation}
which, according to \cite{rockafellar}, Theorem 23.5, is equivalent to $(y,r)\in\partial u(x,q)$. In particular, $\partial u(x,q)\cap\mathcal{L}^*\neq\emptyset$.

Conversely, if $(y,r)\in\partial u(x,q)\cap\mathcal{L}^*$, i.e., if \eqref{eq:equality} holds, then
\begin{align*}
0&\leq\mathbb{E}\left[\int_0^{\hat{T}} V(t,\hat{\delta}_t(y,r))+\hat{c}_t(x,q)\hat{\delta}_t(y,r)-U(t,\hat{c}_t(x,q))d\kappa_t\right]\\
&\leq v(y,r)+xy+qr-u(x,q)=0,
\end{align*}
which immediately implies the relations \eqref{eq:dual-relations-thm}.
Finally, to show that $\partial u(x,q)\subseteq\mathcal{L}^*$, let $(y,r)\in\partial u(x,q)$. Since $\partial u(x,q)$ is a closed convex subset of $\bar{\mathcal{L}}$ and since $\partial u(x,q)\cap\mathcal{L}^*\neq\emptyset$, there is a sequence $(y_n,r_n)$ in $\partial u(x,q)\cap\mathcal{L}^*$ that converges to $(y,r)$. Since each of the sets $\mathcal{D}(y_n,r_n)$ contains the strictly positive process $U'(t,\hat{c}_t(x,q))$, Lemma \ref{lem:limit-in-dual} implies that $(y,r)\in\mathcal{L}^*$.
\end{proof}

%%%%%%%%%%%%%%%%%%%%%%%%%%%%

\section{Complete market}\label{sec:complete}

In this section, we examine the complete market case, when $\mathcal{Z}=\left\{ Z \right\}$ is a singleton, in much greater detail. This becomes possible due to the fact that the description of the dual set $\mathcal{D}^\lambda$ given in Proposition~\ref{prop:min-of-D} simplifies significantly in a complete market:

\begin{prop}\label{prop:complete-case-D}
If $\mathcal{Z}=\left\{ Z \right\}$ is a singleton then
\begin{equation}\label{eq:description-D-complete}
\mathcal{D}^\lambda=\{\delta\geq 0\text{ optional}: \delta\preceq_\lambda Z\},\quad \lambda\in[0,1].
\end{equation}
In particular, for any consumption plan $c$ satisfying \eqref{cond:ddc}, $\sup_{\delta\in\mathcal{D}^\lambda}\langle c,\delta \rangle=\langle c,Z \rangle$.
\end{prop}

\begin{proof}
Clearly, the set on the right-hand side of \eqref{eq:description-D-complete} is $\preceq_\lambda$-solid, convex, and contains $Z$. By Proposition \ref{prop:min-of-D}, it remains to verify that it is closed with respect to convergence in measure $\mathbb{P}\times d\kappa$. Let $\delta^n\preceq_\lambda Z$, $n\geq 1$, be a sequence converging almost surely to an optional process $\delta\geq 0$. To show that $\delta\preceq_\lambda Z$, we check the right-hand side of \eqref{eq:monotonicity-lambda-ord}: for every $c$ satisfying \eqref{cond:ddc}, by Fatou's lemma and by \eqref{eq:monotonicity-lambda-ord} applied to $\delta^n\preceq_\lambda Z$,
$$\langle c, \delta \rangle\leq\liminf_{n\to\infty}\langle c, \delta^n \rangle\leq\langle c, Z\rangle.$$
The last assertion follows from \eqref{eq:monotonicity-lambda-ord} as well.
\end{proof}

We note that from \eqref{eq:description-D-complete} it is easy to show the strict inclusion $\mathcal{D}^{\lambda_2}\supsetneq\mathcal{D}^{\lambda_1}$ for $1\geq\lambda_2>\lambda_1\geq0$ in a complete market. Namely, it suffices to find a process $\delta$ such that $\delta\preceq_{\lambda_2}Z$ but $\delta\npreceq_{\lambda_1}Z$. One can take a stopping time $T\in(0,\hat{T}]$ such that $\mathbb{P}(T<\hat{T})>0$, define
$$\delta_t:=Z_t\mathbbm{1}_{[0,T]}+\delta^2_t,\quad \text{where}\quad\delta^2_t:=\lambda_2\cdot{}^o\left(\frac{\mathbbm{1}_{[0,T]}}{d\kappa([0,T])}\int_{T}^{\hat{T}}Z d\kappa\right)_t,$$
and check the required properties which, by Proposition \ref{prop:lambda-ord-alternative}, are equivalent to $\delta^2\preceq \lambda_2 Z\mathbbm{1}_{(T,\hat{T})}$ and $\delta^2\npreceq \lambda_1 Z\mathbbm{1}_{(T,\hat{T})}$.

\subsection{Case $q=0$}

First, we consider the case $q=0$, when there is no lower bound on consumption -- only the drawdown constraint. As in Section \ref{sec:optimization}, the results hold for any $\lambda\in(0,1]$, unless specified otherwise, and we often omit mentioning the dependence of domains, value functions, and optimizers on $\lambda$. For brevity, we denote $u(x)=u(x,0)$, $\hat{c}(x)=\hat{c}(x,0)$, $v(y)=v(y,0)$, $\hat{\delta}(y)=\hat{\delta}(y,0)$ for $x,y>0$.
\begin{prop}\label{prop:complete-case}
Suppose all the assumptions of Theorem \ref{thm:main-duality} hold and the market is complete. Let $x>0$, $y=u'(x)$, and assume a consumption plan $c$ satisfies \eqref{cond:ddc} and $\langle c,Z \rangle=x$. Then $c\in\mathcal{C}(x)$. The consumption plan $c$ is the optimizer in $\mathcal{C}(x)$ if and only if for $\hat\delta_t:=U'(t,c_t)$ the following holds:
\begin{enumerate}
\item $\{yZ_t>\hat\delta_t\}\subseteq\{c_t=\lambda\bar{c}_t\}$, up to $\mathbb{P}\times d\kappa-$nullsets;
\item $\{\hat\delta_t>yZ_t\}\subseteq\{c_t=\bar{c}_t\}$, up to $\mathbb{P}\times d\kappa-$nullsets;
\item $\mathbb{P}$-almost surely,
\begin{equation*}
\begin{aligned}
^o\left(\int_.^{\hat{T}} (\hat\delta-yZ)\vee0 d\kappa\right)_t\leq{}^o\left(\int_.^{\hat{T}} \lambda(yZ-\hat\delta)\vee 0 d\kappa\right)_t,\quad\text{for all }t\in[0,\hat{T}),\\
\text{with equality }d\bar{c}_t-\text{almost everywhere.}
\end{aligned}
\end{equation*}
In this case, the optimizer  $c$ is related with its running essential supremum $\bar{c}$ in the following way (up to $\mathbb{P}\times d\kappa-$nullsets):
\begin{equation}\label{eq:c-through-c-bar}
c_t=\lambda\bar{c}_t\vee I(t,yZ_t)\wedge \bar{c}_t,\quad t\in[0,\hat{T}),
\end{equation}
and, $\mathbb{P}$-almost surely, $\bar{c}$ satisfies
\begin{equation}\label{eq:cond-on-c-bar}
\begin{aligned}
^o\left(\int_.^{\hat{T}} (U'(\bar{c})-yZ)\vee0 d\kappa\right)_t\leq{}^o\left(\int_.^{\hat{T}} \lambda(yZ-U'(\lambda\bar{c}))\vee 0 d\kappa\right)_t,\quad\text{for all }t\in[0,{\hat{T}}),\\
\text{with equality }d\bar{c}_t-\text{almost everywhere.}
\end{aligned}
\end{equation}
\end{enumerate}
\end{prop}

\begin{rem}[Interpretation]
Let $\lambda\in(0,1)$.We see that only the following three types of behavior are possible for the optimal consumption plan $c=\hat{c}(x)$ (up to $\mathbb{P}\times d\kappa-$nullsets):
 \begin{enumerate}
\item[(i)] $c_t=\lambda\bar{c}_t$. The agent consumes at the minimal level allowed by the drawdown constraint.
\item[(ii)] $c_t=\bar{c}_t$. The agent consumes at the current running essential supremum level.
\item[(iii)] $c_t=I(t,yZ_t)$. The agent consumes as an unconstrained agent with a different initial wealth $x_0$ given by $u_0'(x_0)=y$, where $u_0$ is the value function for the unconstrained problem.
\end{enumerate}
How exactly the timeline separates into these three regions is encoded in \eqref{eq:cond-on-c-bar}, however, this is more difficult to interpret. For the simpler ratchet constraint $\lambda=1$, see Corollary~\ref{cor:complete-env} below and the discussion after it.
\end{rem}

\begin{rem}\label{rem:pts-of-increase}
The jointly measurable process defined by 
$$D_t:=\int_t^{\hat{T}} \left(\hat\delta-yZ\right)\vee0 -\lambda\left(yZ-\hat\delta\right)\vee 0d\kappa,\quad t\in[0,\hat{T}),$$
is continuous. Condition 3.~of Proposition \ref{prop:complete-case} implies $\mathbb{E}[D_0]\leq 0$ and therefore $\mathbb{E}\left[\int_0^{\hat{T}}\hat\delta d\kappa\right]\leq\mathbb{E}\left[\int_0^{\hat{T}} yZ d\kappa\right]$. Hence  $\sup_{t\geq 0} \vert D_t\vert\leq \int_0^{\hat{T}} (\hat\delta+yZ) d\kappa$ is integrable and $D_t$ is uniformly integrable. By \cite{dellacherie-meyerB}, Chapter VI, Theorem no.~47 and Remark no.~50(f), the optional projection ${}^oD_t$ is right-continuous (in fact, c\`adl\`ag). Therefore, by Lemma \ref{lem:pt-of-increase}, condition 3. of Proposition \ref{prop:complete-case} is equivalent to the following: $\mathbb{P}$-almost surely,
\begin{equation*}
\begin{aligned}
^o\left(\int_.^{\hat{T}} (\hat\delta-yZ)\vee0 d\kappa\right)_t\leq{}^o\left(\int_.^{\hat{T}} \lambda(yZ-\hat\delta)\vee 0 d\kappa\right)_t,\quad\text{for all }t\in[0,{\hat{T}}),\\
\text{with equality when } d\bar{c}_t>0.
\end{aligned}
\end{equation*}
The same is true for \eqref{eq:cond-on-c-bar}: ``equality $d\bar{c}_t-$almost everywhere" can be replaced with ``equality when $d\bar{c}_t>0$".
\end{rem}

\begin{proof}[Proof of Proposition \ref{prop:complete-case}]
The consumption plan $c$ belongs to $\mathcal{C}(x)$ because it satisfies \eqref{cond:ddc} and is $x$-admissible by Proposition \ref{prop:complete-case-D}: $\sup_{\delta\in\mathcal{D}^\lambda}\langle c,\delta \rangle=\langle c,Z \rangle=x$.

``$\Rightarrow$''. Assume that $c=\hat{c}(x)$, the optimizer in $\mathcal{C}(x)$. By \cite{mostovyi}, Theorem~3.2, under the assumptions of Theorem \ref{thm:main-duality},
\begin{equation}\label{eq:max-by-duality}
\langle c,\hat\delta \rangle=xy=\langle c,yZ \rangle
\end{equation}
and $\hat\delta$ is the optimizer in $\mathcal{D}(y)$. In particular, $\hat\delta\preceq_\lambda yZ$ and, by  Proposition \ref{prop:lambda-ord-alternative}, $\hat\delta^2\preceq\lambda\delta^2$ for $\hat\delta^2=(\hat\delta-yZ)\vee 0$ and $\delta^2=(yZ-\hat\delta)\vee 0$. Let $\delta^1=\hat\delta\wedge yZ$. Tracing the following sequence of inequalities implies that \eqref{eq:max-by-duality} is only possible if conditions 1.-3. are satisfied:
\begin{equation}\label{eq:sequence}
\begin{aligned}
\langle c,\hat\delta\rangle&=\langle c,\delta^1 \rangle+\langle c,\hat\delta^2 \rangle\leq\langle c,\delta^1 \rangle+\langle \bar{c},\hat\delta^2 \rangle=\langle c,\delta^1 \rangle+\mathbb{E}\left[\int_0^{\hat{T}} {}^o\left(\int_.^{\hat{T}} \hat\delta^2 d\kappa\right)_t d\bar{c}_t\right]\\
&\leq \langle c,\delta^1 \rangle+\mathbb{E}\left[\int_0^{\hat{T}} {}^o\left(\int_.^{\hat{T}} \lambda\delta^2 d\kappa\right)_t d\bar{c}_t\right]=\langle c,\delta^1 \rangle+\langle\bar{c},\lambda\delta^2 \rangle\\
&\leq\langle c,\delta^1 \rangle+\langle c,\delta^2 \rangle=\langle c,yZ \rangle=xy,
\end{aligned}
\end{equation}
where for inequalities we used that $\lambda\bar{c}\leq c\leq\bar{c}$ and $\hat\delta^2\preceq\lambda \delta^2$.

``$\Leftarrow$''. If $\hat\delta$ and $yZ$ satisfy conditions 1.-3. then the above sequence of inequalities \eqref{eq:sequence} holds with equalities everywhere, i.e., \eqref{eq:max-by-duality} holds. Condition 3. implies that $\hat\delta\preceq_\lambda yZ$ and, by \eqref{eq:description-D-complete}, $\hat\delta\in\mathcal{D}(y)$. Using $y=u'(x)$, the conjugacy relations between $u$ (respectively, $U$) and $v$ (respectively, $V$), $c\in\mathcal{C}(x)$, and $\hat\delta\in\mathcal{D}(y)$, we obtain
\begin{align*}
u(x)-v(y)&=xy\overset{\eqref{eq:max-by-duality}}{=}\langle c,\hat\delta \rangle=\langle c_t,U'(t,c_t) \rangle\\
&=\mathbb{E}\left[\int_0^{\hat{T}} U(t,c_t) d\kappa_t\right]-\mathbb{E}\left[\int_0^{\hat{T}} V(t,\hat\delta_t) d\kappa_t\right]\leq u(x)-v(y).
\end{align*}
Hence, the last inequality is in fact an equality, which is only possible when $c$ and $\hat\delta$ are the optimizers in $\mathcal{C}(x)$ and $\mathcal{D}(y)$, respectively.

Finally, we show that the expression \eqref{eq:c-through-c-bar} of $c$ through its running essential supremum and the condition \eqref{eq:cond-on-c-bar} for the latter follow from 1.-3. Since $c\geq\lambda\bar{c}$, we have $\hat\delta=U'(c)\leq U'(\lambda\bar{c})$ and $\{yZ> U'(c)\}\supseteq\{yZ>U'(\lambda\bar{c})\}$. But 1. implies that $\{yZ> U'(c)\}\subseteq\{yZ>U'(\lambda\bar{c})\}$. Hence, $\{yZ> U'(c)\}=\{yZ>U'(\lambda\bar{c})\}$ and on this set $c=\lambda\bar{c}$, $\hat\delta=U'(\lambda\bar{c})$ (everything holds up to $\mathbb{P}\times d\kappa-$nullsets). Similarly, $c\leq\bar{c}$ therefore $U'(c)\geq U'(\bar{c})$ and $\{yZ< U'(c)\}\supseteq\{yZ<U'(\bar{c})\}$. But 2. implies that $\{yZ< U'(c)\}\subseteq\{yZ<U'(\bar{c})\}$, hence $\{yZ< U'(c)\}=\{yZ<U'(\bar{c})\}$ and on this set $c=\lambda\bar{c}$, $\hat\delta=U'(\bar{c})$. On the complement of these two sets, $\hat\delta=yZ$ therefore $c=I(yZ)$. We can summarize this as
\begin{equation*}
c_t=\left\{\begin{aligned}
\lambda\bar{c}_t&\quad\text{on}\quad\{yZ>U'(\lambda\bar{c})\}=\{\lambda\bar{c}>I(yZ)\}, \\
\bar{c}_t&\quad\text{on}\quad\{U'(\bar{c})>yZ\}=\{\bar{c}<I(yZ)\},\\
I(t,yZ_t)&\quad\text{on}\quad \{\bar{c}\geq I(yZ)\geq \lambda\bar{c}\},\\
\end{aligned}\right.
\end{equation*}
or, equivalently, as \eqref{eq:c-through-c-bar}. The condition \eqref{eq:cond-on-c-bar} follows from 3. since $(\hat\delta-yZ)\vee0=(U'(\bar{c})-yZ)\vee 0$ and $(yZ-\hat\delta)\vee 0=(yZ-U'(\lambda\bar{c}))\vee 0$.
\end{proof}

For the ratchet constraint, the characterization of the optimizers given in Proposition \ref{prop:complete-case} simplifies significantly.
\begin{cor}[$\lambda=1$]\label{cor:complete-env}
Suppose all the assumptions of Theorem \ref{thm:main-duality} hold and the market is complete. Let $x>0$, $y=u'(x)$, and $c\in\mathcal{C}_{\text{inc}}$ such that $\langle c,Z \rangle=x$. Then $c\in\mathcal{C}^1(x)$. The consumption plan $c$ is the optimizer in $\mathcal{C}^1(x)$ if and only if, $\mathbb{P}-$almost surely,
\begin{equation}\label{eq:env-ratchet}
^o\left(\int_.^{\hat{T}} U'(c)d\kappa\right)_t\leq{}^o\left(\int_.^{\hat{T}} yZ d\kappa\right)_t\quad\text{for all }t\in[0,{\hat{T}}),\ \text{with equality when }dc_t>0.
\end{equation}
\end{cor}
This characterization is closely related to the notion of the \textit{envelope process} introduced in Lemma A.1 in \cite{BK} and to the Representation Theorem of \cite{bank-el-karoui} it is based on. In Lemma \ref{lem:envelope} of the \hyperref[app:envelope]{Appendix}, we slightly modify Lemma~A.1 of \cite{BK} to show the following: if, in addition to Assumption~\ref{ass:utility} on utility $U$, we assume that $\mathbb{E}\left[\int_0^{\hat{T}}U'(\omega,t,x)d\kappa_t\right]<\infty$ for every $x>0$ then for every $y>0$ there exists a unique $c=c^y\in\mathcal{C}_\text{inc}$ for which \eqref{eq:env-ratchet} holds. By Corollary \ref{cor:complete-env}, this consumption plan has to be the optimizer, $c^y=\hat{c}(x)\in\mathcal{C}^1(x)$. In the \hyperref[app:envelope]{Appendix}, we also give a short proof, inspired by the arguments of \cite{riedel} and not involving duality, that, under an additional assumption on utility (Assumption \ref{ass:utility-additional}), the optimal consumption plans have this structure.

In the example below, we derive, as a special case of Corollary \ref{cor:complete-env}, the formula of \cite{riedel} for the optimal consumption plans.

\begin{exmp}
Let $\hat{T}=\infty$ and assume that stochastic clock is given by $\dot\kappa_s=e^{-rs}$, where $r>0$ is the interest rate (cf.~Remark \ref{rem:on-model}(ii)). We take the utility field given by $U(t,x)=\frac{e^{-\delta t} \mathtt{u}(x)}{e^{-rt}}$, where $\mathtt{u}$ is a strictly concave, increasing, continuously differentiable deterministic function on $(0,\infty)$ satisfying the Inada conditions $\mathtt{u}(0)=+\infty$, $\mathtt{u}(+\infty)=0$ and $\delta>0$ is the parameter of exponential time preferences of the agent. This utility satisfies Assumption \ref{ass:utility} and the expected utility functional in \eqref{eq:primal-problem} becomes
$\mathbb{E}\left[\int_0^\infty e^{-\delta t} \mathtt{u}(c_t)dt\right]$. Assume further that $\log(Z_t)$ is a L\'evy process (starting at zero) for the unique equivalent martingale deflator $Z$. With these choices of $\hat{T}$, $\kappa$, $U$, and $Z$, we turn out in the framework of \cite{riedel}. The case of the GBM market considered in \cite{dybvig} corresponds to $\log(Z_t)=-\theta B_t-\frac{1}{2}\theta^2 t$, where $B_t$ is the underlying Brownian motion and $\theta$ is the market price of risk.

Let $\mathtt{i}=(\mathtt{u}')^{-1}$. We will show that for a suitable constant $K>0$ the consumption plan
\begin{equation}\label{eq:riedel-sol}
c_t:=\begin{cases}0,& t=0,\\ \mathtt{i}\left(\inf_{s\in[0,t)}KZ_se^{(\delta-r)s}\right),& t>0,\end{cases}\quad\in\mathcal{C}_\text{inc}\end{equation}
satisfies \eqref{eq:env-ratchet}. By Lemma \ref{lem:envelope}, this is then the unique process in $\mathcal{C}_\text{inc}$ satisfying \eqref{eq:env-ratchet}, hence by Corollary \ref{cor:complete-env} it is the optimizer $\hat{c}(x)\in\mathcal{C}^1(x)$. The definition \eqref{eq:riedel-sol} corresponds to the formula (3) of \cite{riedel} for the optimal consumption in case $q=0$; in order to obtain the result analogous to Ridel's for the case $q>0$, we can simply apply Corollary \ref{cor:q-positive} below.

With the martingale property of $Z$ and with $\dot\kappa$ being deterministic and exponential, it is easy to check for the right-hand side of \eqref{eq:env-ratchet} that
${}^o\left(\int_.^\infty yZ d\kappa\right)_t=yZ_t\cdot d\kappa([t,\infty))=y\frac{e^{-rt}}{r}Z_t$.

With $c$ defined as in \eqref{eq:riedel-sol}, we have
\begin{equation}\label{eq:before-opt-proj}
\begin{aligned}
\int_t^\infty U'(c)d\kappa&=\int_t^\infty e^{-\delta s}\left(\inf_{u\in[0,s)} KZ_u e^{(\delta-r)u}\right) ds\leq \int_t^\infty e^{-\delta s}\left(\inf_{u\in[t,s)} KZ_u e^{(\delta-r)u}\right) ds\\
&=Ke^{-rt}Z_t\cdot\int_0^\infty e^{-\delta s} \left(\inf_{u\in[0,s)} \frac{Z_{t+u}}{Z_t} e^{(\delta-r)u}\right) ds=:Ke^{-rt}Z_t I_t,
\end{aligned}
\end{equation}
where ``$=$" holds in place of ``$\leq$" on $\{dc_t>0\}\in\mathcal{O}$. Due to the L\'evy assumption on $\log(Z_t)$, the integral $I_t$ satisfies $\mathbb{E}[I_t\vert\mathcal{F}_t]=\mathbb{E}\left[\int_0^\infty e^{-\delta s}\inf_{u\in[0,s)} Z_u e^{(\delta-r)u} ds\right]:=I\in(0,\infty)$, hence ${}^oI_t\equiv I$. Taking the optional projection in \eqref{eq:before-opt-proj}, we obtain
$$^o\left(\int_.^\infty U'(c)d\kappa\right)_t\leq Ke^{-rt}Z_tI\quad\text{for all }t\geq 0,\ \text{with equality when }dc_t>0.$$
This is precisely \eqref{eq:env-ratchet} if we take $K:=y/(Ir)$.
\end{exmp}

%%%%%%%%%%%%%%%%%%%%
%%%%%%%%%%%%%%%%%%%%

\subsection{Monotonicity and continuity of optimizers}

Next, we establish monotonicity and continuity of the optimizers with respect to the initial wealth. These results will be useful when dealing with the case $q>0$ in the following subsection.
\begin{prop}\label{prop:mon-and-cont}
Suppose all the assumptions of Theorem \ref{thm:main-duality} hold and the market is complete. For $x_2>x_1>0$,
$$\hat{c}(x_1)\leq \hat{c}(x_2),\quad \mathbb{P}\times d\kappa-\text{a.e.},$$
and for $x_n\to x$ with $x,x_n>0$, $\hat{c}(x_n)\to \hat{c}(x)$, $\mathbb{P}\times d\kappa-$almost everywhere.
\end{prop}

\begin{proof}
For fixed $x_2>x_1>0$ we denote for brevity $y_i=u'(x_i)$, $c^i=\hat{c}(x_i)$ and $\bar{c}^i$ its running essential supremum, $i=1,2$. By \eqref{eq:c-through-c-bar}, $c^i=\lambda\bar{c}^i\vee I(y_iZ)\wedge \bar{c}^i$ for $i=1,2$. We split the product space $\Omega\times[0,\hat{T})$ into two regions:
$$R:=\left\{(\omega,t):\ \bar{c}_t^1\geq I(t,y_1Z_t)\geq \lambda\bar{c}^1_t\text{ and } \bar{c}^2_t\geq I(t,y_2Z_t)\geq \lambda\bar{c}^2_t\right\}\in\mathcal{O}$$
and its complement. On $R$, $c^1_t=I(t,y_1Z_t)<I(t,y_2Z_t)=c^2_t$ since $y_1>y_2$ and $I$ is strictly decreasing. In order to prove $c^1\leq c^2$, it remains to show that $\bar{c}^1\leq\bar{c}^2$ on $R^c$, $\mathbb{P}\times d\kappa-$a.e.

Let $f$ be a (time-dependent and stochastic) function defined by
$$f(y,\bar{c}):=f(\omega, t,y,\bar{c}):=(U'(t,\bar{c})-yZ_t)\vee0 -\lambda(yZ_t-U'(t,\lambda\bar{c}))\vee 0\quad\text{for}\quad y,\bar{c}>0.$$
and let
$$D^i_t:=\int_t^{\hat{T}} f\left(s,y_i,\bar{c}^i_s\right) d\kappa_s=\int_t^{\hat{T}} \left(\hat\delta(y_i)-y_iZ\right)\vee0 -\lambda\left(y_iZ-\hat\delta(y_i)\right)\vee 0d\kappa,\quad i=1,2,$$
as in Remark \ref{rem:pts-of-increase} (the second equality holds up to indistinguishability). By Remark \ref{rem:pts-of-increase},
\begin{equation}\label{eq:eq-at-pts-increase}
\mathbb{P}-\text{almost surely}:\quad ^oD^i_t\leq 0\quad\text{for all }t\in [0,\hat{T})\text{ with }``="\text{ if } d\bar{c}^i_t>0.
\end{equation}
For a fixed $l\geq 0$, we define two stopping times
${T}^i_l:=\inf\{t\in[0,\hat{T}): \bar{c}^i_t> l \}$, $i=1,2$, where the infimum of an empty set in taken to be $\hat{T}$. As in the proof of Lemma \ref{lem:pt-of-increase}, ${T}^i_l$ is either $\hat{T}$, or a point of increase of $\bar{c}^i$. Next, let
\begin{equation*}
S^1_l:=\begin{cases}
{T}^1_l,&\quad\text{if }{T}^1_l<{T}^2_l,\\
\hat{T},&\quad\text{otherwise},
\end{cases}\quad\text{and}\quad S^2_l:=
\begin{cases}
{T}^2_l,&\quad\text{if }{T}^1_l<{T}^2_l,\\
\hat{T},&\quad\text{otherwise}.
\end{cases}
\end{equation*}
Thus, $S^1_l={T}^1_l<{T}^2_l=S^2_l$ on $\{\omega: {T}^1_l(\omega)<{T}^2_l(\omega)\}$ and $S^1_l=S^2_l=\hat{T}$ otherwise. Property \eqref{eq:eq-at-pts-increase} implies
\begin{equation*}
\mathbb{E} D^1_{S_l^1}=0,\quad \mathbb{E} D^1_{S_l^2}\leq 0,\quad \mathbb{E}D^2_{S_l^1}\leq 0,\quad\text{and}\quad \mathbb{E} D^2_{S_l^2}=0.
\end{equation*}
Taking into account $S^1_l\leq S^2_l$ and the definition of $D^i$, $i=1,2$, we obtain:
\begin{equation}\label{eq:ineq-at-stopping}
\mathbb{E}\left[\int_{S^1_l}^{S^2_l}f\left(y_1,\bar{c}^1\right) d\kappa\right]\geq 0\geq \mathbb{E}\left[\int_{S^1_l}^{S^2_l}f\left(y_2,\bar{c}^2\right) d\kappa\right].
\end{equation}
On the other hand, $\bar{c}^1>l\geq \bar{c}^2$ on $(S^1_l,S^2_l]$ and $y_1>y_2$, so by the monotonicity of $f$,
\begin{equation}\label{eq:mon-of-f}
f\left(\omega,t,y_1,\bar{c}^1_t(\omega)\right)\leq f\left(\omega,t,y_2,\bar{c}^2_t(\omega)\right)\quad \text{on }(S^1_l(\omega),S^2_l(\omega)].
\end{equation}
Moreover, the inequality in \eqref{eq:mon-of-f} is strict on $R^c$ due to the strict monotonicity of $f$ on $R^c$ for at least one of $\bar{c}^1$ or $\bar{c}^2$: either $\bar{c}_t^1\in \left[I(t,y_1Z_t),\frac{1}{\lambda}I(t,y_1Z_t)\right]^c$ and then $$f\left(\omega,t,y_1,\bar{c}^1_t(\omega)\right)<f\left(\omega,t,y_1,\bar{c}^2_t(\omega)\right)\leq f\left(\omega,t,y_2,\bar{c}^2_t(\omega)\right),$$ or $\bar{c}^2_t\in \left[I(t,y_2Z_t),\frac{1}{\lambda}I(t,y_2Z_t)\right]^c$ and then $$f\left(\omega,t,y_2,\bar{c}^2_t(\omega)\right)> f\left(\omega,t,y_2,\bar{c}^1_t(\omega)\right)\geq f\left(\omega,t,y_1,\bar{c}^1_t(\omega)\right).$$
Therefore, \eqref{eq:ineq-at-stopping} implies that the set
 $N_l:=\{(\omega,t): S^1_l(\omega)<t\leq S^2_l(\omega)\}\cap R^c$
is a $\mathbb{P}\times d\kappa-$nullset.

If $\{(\omega,t):\bar{c}_t^1>\bar{c}_t^2\}\cap R^c$ has a positive $\mathbb{P}\times d\kappa$ measure then there exists an $l\in\mathbb{Q}$ such that the set
$$\{(\omega,t):\bar{c}_t^1> l\geq \bar{c}_t^2\}\cap R^c=\{(\omega,t):S^1_l(\omega)<t\leq S^2_l(\omega)\}\cap R^c= N_l$$
has a positive $\mathbb{P}\times d\kappa$ measure, a contradiction. Therefore, $\bar{c}_t^1\leq\bar{c}_t^2$, $\mathbb{P}\times d\kappa-$a.e. on $R^c$, completing the proof of monotonicity.

To prove continuity, we take an increasing sequence $x_n\uparrow x$ (for a decreasing sequence the argument is analogous). By $\lim_{n\to\infty}\hat{c}(x_n)\leq \hat{c}(x)$, $\mathbb{P}\times d\kappa-$a.e., and by monotone convergence
$$x=\lim_{n\to\infty} x_n=\lim_{n\to\infty}\langle\hat{c}(x_n),Z\rangle = \langle\lim_{n\to\infty}\hat{c}(x_n),Z\rangle\leq \langle\hat{c}(x),Z\rangle=x,$$
so the equality should hold in place of inequality, i.e., $\hat{c}(x)=\lim_{n\to\infty}\hat{c}(x_n)$, $\mathbb{P}\times d\kappa-$a.e.
\end{proof}

%%%%%%%%%%%%%%%%%%%%%%%%%%%%
%%%%%%%%%%%%%%%%%%%%%%%%%%%%

\subsection{Case $q>0$}
Now we consider the presence of a lower bound $q>0$ on initial consumption. It turns out that optimizers for $q>0$ can be described in terms of appropriate optimizers for $q=0$, as the following proposition states. Note that in the case of a complete market the constant $\alpha$ defined in Proposition \ref{prop:alpha-for-D} becomes simply $\alpha=\mathbb{E}\left[\int_0^{\hat{T}} Zd\kappa\right]$.

\begin{prop}
Suppose all the assumptions of Theorem \ref{thm:main-duality} hold and the market is complete. Fix $q>0$. For $x>0$, $y=u'(x)$, let $\bar{c}^{x}$ be the running essential supremum of the optimizer $\hat{c}(x)\in\mathcal{C}(x)$ and define
\begin{equation}\label{eq:lower-bound}
c_t:=c_t(x,q)=\lambda [\bar{c}_t^{x}\vee q]\vee I(t,yZ_t)\wedge [\bar{c}_t^{x}\vee q]=\hat{c}(x)\vee[\lambda q\vee I(t,yZ_t)\wedge q]
\end{equation}
for $t\in[0,\hat{T})$ and $\pi(x):=\langle c,Z\rangle$, the price of the consumption plan $c$. Then $\pi(x)\in[(\alpha\lambda)q,\infty)$, and if $\pi(x)>(\alpha\lambda)q$ then $c\in\mathcal{C}(\pi(x),q)$ and $c$ is the optimizer in $\mathcal{C}(\pi(x),q)$.

The function $x\mapsto \pi(x)$ is continuous non-decreasing on $(0,\infty)$ with $\pi(x)\downarrow (\alpha\lambda)q$ when $x\downarrow 0$ and $\pi(x)\uparrow \infty$ when $x\uparrow \infty$. As a consequence, for every $x'>(\alpha\lambda)q$ there exists an $x>0$ such that $\pi(x)=x'$ and hence the optimizer $\hat{c}(x',q)\in\mathcal{C}(x',q)$ is given by \eqref{eq:lower-bound}.
\end{prop}

\begin{rem}[Interpretation]
In view of Proposition \ref{prop:complete-case}, $\hat{c}(x)=\lambda\bar{c}^x\vee I(yZ)\wedge \bar{c}^x$, therefore, the optimizer in $\mathcal{C}(x',q)$, $(x',q)\in\mathcal{K}$, can be described as follows: for some $x>0$, the optimizers $\hat{c}(x)$ and $\hat{c}(x',q)$ behave identically on $\{\bar{c}^x\geq q\}$, i.e., starting from the time when the running essential supremum of $\hat{c}(x)$ becomes at least $q$. On $\{\bar{c}^x< q\}$, $\hat{c}(x',q)$ behaves as an unconstrained agent's consumption $I(t,yZ_t)$ restricted to stay between $\lambda q$ and $q$: $\hat{c}(x',q)=\lambda q\vee I(t,yZ_t)\wedge q$. This agrees with the observation that the drawdown constraint \eqref{cond:ddc-with-q} with $q>0$ becomes the simple drawdown constraint \eqref{cond:ddc} on the set $\{\bar{c}\geq q\}$ and becomes simply a lower bound $c\geq \lambda q$ on $\{\bar{c}< q\}$.

As in the case $q=0$, we have three possible types of behaviour: either the agent consumes at the minimal level allowed by the drawdown constraint \eqref{cond:ddc-with-q}, $c_t=\lambda[\bar{c}_t\vee q]=\lambda [\bar{c}_t^{x}\vee q]$, or the agent consumes at the current running essential supremum level $c_t=\bar{c}_t=\bar{c}_t^{x}\vee q$, or the agent consumes as an unconstrained agent, $c_t=I(t,yZ_t)$. This separation into three regions agrees with the result of \cite[Theorem 1]{Arun2012} on the drawdown constraint in a GBM market with CRRA utility, but it does not seem obvious from Arun's result that the optimal consumption can be linked to the optimal consumption of an unconstrained agent.

We note that $\pi(x)=(\alpha\lambda)q$ is only possible if $c$ given by \eqref{eq:lower-bound} satisfies $c\equiv\lambda q$, $\mathbb{P}\times d\kappa-$a.e. But this is the only admissible consumption plan satisfying \eqref{cond:ddc-with-q} with initial wealth $(\alpha\lambda)q$ and therefore $c$ given by \eqref{eq:lower-bound} is the optimizer in this trivial case as well. 
\end{rem}

\begin{proof} First, we prove the announced properties of the function $x\mapsto\pi(x)$. Since $x+\alpha q\geq\langle\hat{c}(x)+ q,Z\rangle\geq \pi(x)\geq\langle \lambda q,Z\rangle=(\alpha\lambda)q$, we have
$\pi(x)\in[(\alpha\lambda)q,\infty)$. Since $c\geq \hat{c}(x)$, we also obtain $\pi(x)=\langle c, Z\rangle\geq \langle \hat{c}(x),Z\rangle =x$, i.e., $\pi(x)\uparrow\infty$ as $x\uparrow \infty$. The map $x\mapsto\hat{c}(x)$ is continuous and monotone in the sense of Proposition \ref{prop:mon-and-cont}, moreover, by this monotonicity $\hat{c}(x)\downarrow 0$ as $x\downarrow 0$, $\mathbb{P}\times d\kappa-$a.e., since $\langle \hat{c}(x),Z\rangle=x\downarrow 0$ and $Z>0$, $\mathbb{P}\times d\kappa-$a.e. The function $x\mapsto y=u'(x)$ is strictly decreasing with $u'(0)=\infty$ by \cite[Theorem 3.2]{mostovyi}, while the (random, time-dependent) function $I$ is strictly decreasing with $I(\infty)=0$. Therefore, for every $(\omega,t)$ the map $x\mapsto I(t,yZ_t)$ is continuous and strictly increasing with $I(t,yZ_t)\downarrow 0$ as $x\downarrow 0$. These observations and \eqref{eq:lower-bound} imply continuity and monotonicity of $x\mapsto c(x,q)$ in the sense of Propostion \ref{prop:mon-and-cont}. By monotone convergence $x\mapsto \pi(x)=\langle c(x,q),Z\rangle$ is continuous and non-decreasing on $(0,\infty)$, and $\pi(x)\downarrow \langle 0\vee[\lambda q\vee 0\wedge q],Z\rangle=(\alpha\lambda)q$ as $x\downarrow 0$.

Now, let $x':=\pi(x)>(\alpha\lambda)q$ for some $x>0$. The consumption plan $c$ defined by \eqref{eq:lower-bound} satisfies \eqref{cond:ddc} because it is bounded between the non-decreasing processes $\bar{c}_t^{x}\vee q$ and $\lambda[\bar{c}_t^{x}\vee q]$. Hence, $c$ is $x'$-admissible, satisfies \eqref{cond:ddc} and $c\geq\lambda q$, meaning $c\in\mathcal{C}(x',q)$. It remains to show that $c$ is the optimizer in $\mathcal{C}(x',q)$.

Define $\hat\delta_t=\hat\delta_t(y)=U'(t,\hat{c}_t(x))$ and
\begin{equation}\label{eq:def-of-delta}
\tilde\delta_t=U'(t,c_t)=U'(\lambda [\bar{c}_t^{x}\vee q])\wedge yZ_t\vee U'(\bar{c}_t^{x}\vee q).
\end{equation}
Since $c\geq\hat{c}(x)$, we have $\tilde\delta\leq \hat{\delta}\in\mathcal{D}(y)$ and therefore $\tilde\delta\preceq_\lambda yZ$ by \eqref{prop:complete-case-D}. Furthermore, by Proposition \ref{prop:sufficient-for-Dyr2}, $\tilde\delta\in\mathcal{D}(y,r)$ with $r$ given by
\begin{equation*}
\begin{aligned}
r:&=\mathbb{E}\left[\int_0^{\hat{T}} (\tilde\delta-yZ)\vee 0-\lambda(yZ-\tilde\delta)\vee0 d\kappa\right]\\
&=\mathbb{E}\left[\int_0^{\hat{T}} (U'(\bar{c}_t^{x}\vee q)-yZ)\vee 0-\lambda(yZ-U'(\lambda [\bar{c}_t^{x}\vee q]))\vee0 d\kappa\right].
\end{aligned}
\end{equation*}
We will finish the proof by showing that $\langle c,\tilde\delta\rangle=x'y+qr$. This equality, $c\in\mathcal{C}(x',q)$, $\tilde\delta\in\mathcal{D}(y,r)$, and conjugacy relations between $U$ and $V$ then yield
$$u(x',q)\geq\mathbb{E}\left[\int_0^{\hat{T}} U(t,c_t)d\kappa_t\right]\overset{\eqref{eq:def-of-delta}}{=}\mathbb{E}\left[\int_0^{\hat{T}} V(t,\tilde\delta_t)d\kappa_t\right]+x'y+qr\geq v(y,r)+x'y+qr,$$
which by \eqref{eq:conjugacy-rel} implies $c=\hat{c}(x',q)$ and $\tilde\delta=\hat\delta(y,r)$.

Let $c':=\bar{c}^{x}\vee q-q\in\mathcal{C}_{\text{inc}}$. The equality $\langle c,\tilde\delta\rangle=x'y+qr$ is equivalent to $\langle c, \tilde\delta - yZ\rangle =qr$, which can be checked with the following sequence of equalities:
\begin{equation*}
\begin{aligned}
\langle c, \tilde\delta - yZ\rangle \overset{\eqref{eq:def-of-delta}}{=}& \langle c\cdot\mathbbm{1}_{\{I(yZ)>\bar{c}^{x}\vee q\}},(U'(\bar{c}^{x}\vee q)-yZ)\vee 0\rangle\\
&\quad\quad-\langle c\cdot\mathbbm{1}_{\{I(yZ)<\lambda[\bar{c}^{x}\vee q]\}},(yZ-U'(\lambda [\bar{c}^{x}\vee q]))\vee0\rangle\\
\overset{\eqref{eq:lower-bound}}{=}&\langle \bar{c}^{x}\vee q,(U'(\bar{c}^{x}\vee q)-yZ)\vee 0\rangle-\langle \lambda[\bar{c}^{x}\vee q],(yZ-U'(\lambda [\bar{c}^{x}\vee q]))\vee0\rangle\\
=&\ qr+\langle c', (U'(\bar{c}^{x}\vee q)-yZ)\vee 0-\lambda (yZ-U'(\lambda [\bar{c}^{x}\vee q]))\vee0\rangle\\
=&\ qr+\langle c', (U'(\bar{c}^{x})-yZ)\vee 0-\lambda (yZ-U'(\lambda\bar{c}^{x}))\vee0\rangle\\
=&\ qr + \mathbb{E}\left[\int_0^{\hat{T}}{}^o\left(\int_.^{\hat{T}} (U'(\bar{c}^x)-yZ)\vee0 - \lambda(yZ-U'(\lambda\bar{c}^x))\vee 0 d\kappa\right)_tdc'_t\right]=qr,
\end{aligned}
\end{equation*}
where the expectation on the last line vanishes by \eqref{eq:cond-on-c-bar} due to the fact that the measure $dc'$ is absolutely continuous with respect to $d\bar{c}^x$.
\end{proof}

\begin{cor}[$\lambda=1$]\label{cor:q-positive}
Suppose all the assumptions of Theorem \ref{thm:main-duality} hold and the market is complete. Fix $q>0$. For $x>0$, define $c:=\hat{c}(x)\vee q$ and $\pi(x):=\langle c,Z\rangle$, the price of the consumption plan $c$. Then $\pi(x)\in[\alpha q,\infty)$, and if $\pi(x)>\alpha q$ then $c\in\mathcal{C}^1(\pi(x),q)$ and $c$ is the optimizer in $\mathcal{C}^1(\pi(x),q)$. Moreover, for every $x'>\alpha q$ there exists an $x>0$ such that $\pi(x)=x'$ and hence the optimizer $\hat{c}(x',q)\in\mathcal{C}^1(x',q)$ is given by $\hat{c}(x)\vee q$.
\end{cor}

In Proposition \ref{prop:alternative-sol} in the \hyperref[app:envelope]{Appendix}, we achieve a similar conclusion for the ratchet constraint without relying on Theorem \ref{thm:main-duality}, but under additional assumptions on $U$, by using the Bank--Kauppila envelope process.

%%%%%%%%%%%%%%%%%%%%%%%%%%%%%%%%%%%%%%%%%%%%%%
%% Single Appendix:                         %%
%%%%%%%%%%%%%%%%%%%%%%%%%%%%%%%%%%%%%%%%%%%%%%
\begin{appendix}

%\section*{???}%% if no title is needed, leave empty \section*{}.
\section*{Ratcheting consumption in complete market -- envelope process and alternative solution}\label{app:envelope}

The following lemma is a modification of \cite[Lemma A.1]{BK} for less restrictive assumptions on $U$ and possibly finite time horizon $\hat{T}$. The proof is based on the Representation Theorem of \cite{bank-el-karoui} and follows closely the proof of \cite[Lemma A.1]{BK}. This lemma is used for showing existence and uniqueness of $c\in\mathcal{C}_{\text{inc}}$ satisfying \eqref{eq:env-ratchet} (see the discussion following Corollary \ref{cor:complete-env}), as well as for an alternative solution of the ratchet constraint problem in a complete market given in Proposition \ref{prop:alternative-sol} below.

In a complete market, $\mathcal{Z}=\left\{Z\right\}$ is a singleton and we denote $\tilde{Z}_t:=\int_t^{\hat{T}} Zd\kappa$ for $0\leq t\leq \hat{T}$.

\begin{lem}\label{lem:envelope}
Assume that stochastic clock satisfies Assumption~\ref{ass:clock}, utility satisfies Assumption~\ref{ass:utility}, and $\mathbb{E}\left[\int_0^{\hat{T}}U'(\omega,t,x)d\kappa_t\right]<\infty$ for every $x>0$. Then for every $y>0$ there exists a unique process $c^y\in\mathcal{C}_{\text{inc}}$ such that, $\mathbb{P}$-almost surely,
\begin{equation}\label{eq:envelope2}
^o\left(\int_.^{\hat{T}} U'(c^y) d\kappa\right)_t\leq y{}^o\tilde{Z}_t\quad\text{for all}\quad t\in[0,\hat{T}),\ \text{with equality when } dc^y_t>0.
\end{equation}
The process $c^y$ is finite-valued.
\end{lem}
In the language of \cite{BK}, $^o\left(\int_.^{\hat{T}} U'(c^y) d\kappa\right)$ is called the \textit{envelope process} of $y{}^o\tilde{Z}$, and Lemma \ref{lem:envelope} states its existence and uniqueness under given assumptions on $U$.
\begin{proof}
Let $$f(\omega,t,l):=\begin{cases}U'(\omega,t,-1/l),&\quad l<0,\\
-l,&\quad l\geq0.\end{cases}$$
By the properties of $U$, this mapping satisfies:
\begin{itemize}
\item For every $(\omega,t)\in\Omega\times[0,\hat{T})$, the function $l\mapsto f(\omega,t,l)$ is continuous and strictly decreasing from $+\infty$ to $-\infty$.
\item For all $l\in\mathbb{R}$, $(\omega,t)\mapsto f(\omega,t,l)$ is an optional process with $\mathbb{E}\left[\int_0^{\hat{T}}\vert f(\omega,t,l)\vert d\kappa_t\right]<\infty$.
\end{itemize}
Since $\tilde{Z}$ is a continuous non-increasing jointly measurable process with $Z_{\hat{T}}=0$ and $\mathbb{E}[\tilde{Z}_0]=\alpha<\infty$, the optional projection ${}^o\tilde{Z}$ is a class (D), continuous in expectation, non-negative supermartingale with $\lim_{t\uparrow \hat{T}}{}^o\tilde{Z}_{t}=0$, $\mathbb{P}-$a.s.
By Theorem 3 and Remark 2.1 in \cite{bank-el-karoui} there exists an optional process $L_t$ taking values in $\mathbb{R}\cup\{-\infty\}$ such that for every stopping time $S\leq\hat{T}$,
\begin{equation}\label{eq:repr-formula2}
\mathbb{E}\left[\int_S^{\hat{T}} f(t,\sup_{s\in[S,t)}L_s)d\kappa_t\Big\vert \mathcal{F}_S\right]=y{}^o\tilde{Z}_S.
\end{equation}
The process $L_t$ is non-positive on $[0,\hat{T})$ up to indistinguishability. Assuming otherwise, by the optional section theorem there exists a stopping time $S\leq\hat{T}$ such that $\mathbb{P}(S<\hat{T})>0$ and $L_S>0$ on $\{S<\hat{T}\}$, which implies
$$0\leq y{}^o\tilde{Z}_S\leq \mathbb{E}\left[\int_S^{\hat{T}} f(t,L_S)d\kappa_t\Big\vert \mathcal{F}_S\right]=-L_S\mathbb{E}\left[d\kappa([S,\hat{T}))\Big\vert \mathcal{F}_S\right]<0,$$
a contradiction. Furthermore, almost every path of the non-positive left-continuous process $\tilde{L}_t:=\sup_{s\in[0,t)}L_s$ is strictly negative on $(0,\hat{T})$: otherwise, for the stopping times
$$S:=\inf\{t\in(0,\hat{T}]: \tilde{L}_t=0\}\quad\text{and}\quad S^n:=\inf\{t\in(0,\hat{T}]:\tilde{L}_t>-1/n\},\ n\geq 1,$$
 we have $\mathbb{P}(S<\hat{T})>0$, $S^n\uparrow S$ as $n\uparrow\infty$, $\sup_{s\in[S^n,t)}L_s=\tilde{L}_t$ for every $n\geq 1$, $t>S^n$, and, by monotone convergence,
$$0=\lim_{n\to\infty}\mathbb{E}\left[\int_{0}^{\hat{T}} f(t,\tilde{L}_t)\mathbbm{1}_{[S^n,\hat{T})}d\kappa_t\right]=\lim_{n\to\infty}\mathbb{E}\left[y\tilde{Z}_{S^n}\right]=\mathbb{E}\left[y\tilde{Z}_{S}\right]>0.$$
Hence, the process
$$c^y_t:=\begin{cases}
0,&\quad t=0,\\
-1/\tilde{L}_t,&\quad 0<t<\hat{T},
\end{cases}$$
belongs to $\mathcal{C}_{\text{inc}}$, is finite-valued, and, by \eqref{eq:repr-formula2}, satisfies \eqref{eq:envelope2}.

For uniqueness, assume $c\in\mathcal{C}_{\text{inc}}$ is such that $^o\left(\int_.^{\hat{T}} U'(c) d\kappa\right)_t\leq y{}^o\tilde{Z}_t$ for all $0\leq t<\hat{T}$ with ``=" when $dc_t>0$. For $l>0$, we show below that $T_l:=\inf\{t\geq 0: c_t>l\}$ is the largest stopping time minimizing $\mathbb{E}\left[\int_T^{\hat{T}} (yZ-U'(l))d\kappa\right]$ over all stopping times $T\in[0,\hat{T}]$. This means that the stopping times $T_l$, $l>0$, are uniquely determined, hence, $c$ is uniquely determined, and coincides with $c^y$, up to indistinguishability. Note that $\mathbb{E}\left[\int_0^{\hat{T}}\vert yZ-U'(l)\vert d\kappa\right]<\infty$ due to the assumptions on $\kappa$ and $U'$. For a stopping time $T\in[0,\hat{T}]$, we have
$$\mathbb{E}\left[\int_T^{\hat{T}}(yZ-U'(l)) d\kappa\right]\geq \mathbb{E}\left[\int_T^{\hat{T}}(U'(c)-U'(l)) d\kappa\right]\geq\mathbb{E}\left[\int_{T_l}^{\hat{T}}(U'(c)-U'(l)) d\kappa\right],$$
where the first inequality follows from $^o\left(\int_.^{\hat{T}} U'(c) d\kappa\right)_t\leq y{}^o\tilde{Z}_t$, the second inequality follows from the definition of $T_l$ and monotonicity of $U'$, and equality holds in both for $T=T_l$. Thus, $T_l$ is a solution to the optimal stopping problem under consideration. Moreover, any stopping time $T$ such that $\mathbb{P}(T>T_l)>0$ will yield a strict inequality between the second and third terms, making $T_l$ the largest solution to this optimal stopping problem.
\end{proof}

The solution below for the ratchet constraint in a complete semimartingale market is inspired by the argument of \cite{riedel} for a complete market with pricing kernel generated by a L\'evy process. It turns out that the notion of the envelope process allows for such a generalization of Riedel's result. A similar argument, but again only for the L\'evy market model, can be found in \cite{Watson-Scott} in their use of the Bank--El Karoui Representation Theorem.

We denote $\mathbb{U}(c):=\mathbb{E}\left[\int_0^{\hat{T}} U(t,c_t)d\kappa_t\right]$, the expected utility functional for a consumption plan $c$, and make additional assumptions on the utility $U$:
\begin{assume}\label{ass:utility-additional}
$\mathbb{E}\left[\int_0^{\hat{T}}U^{-}(\omega,t,0)d\kappa_t\right]<\infty$ and $\mathbb{E}\left[\int_0^{\hat{T}}U'(\omega,t,x)d\kappa_t\right]<\infty$ for all $x>0$.
\end{assume}

\begin{prop}\label{prop:alternative-sol}
Let $\kappa$ satisfy Assumption \ref{ass:clock}, $U$ satisfy Assumptions \ref{ass:utility} and \ref{ass:utility-additional}, assume $\mathcal{Z}=\{Z\}$ is a singleton, and let $y>0$, $q\geq 0$. Define the consumption plan $\hat{c}:=c^y\vee q\mathbbm{1}_{(0,\hat{T})}\in\mathcal{C}_\text{inc}$, where $c^y$ is given by Lemma \ref{lem:envelope}, and let $x:=\langle\hat{c},Z\rangle$. If $x<\infty$ and $\mathbb{U}(\hat{c})<\infty$ then $\hat{c}$ is the unique maximizer of $\mathbb{U}$ in $\mathcal{C}^1(x,q)$.
\end{prop}

\begin{proof}
Let $c\in\mathcal{C}_{\text{inc}}$ be $x$-admissible and such that $c_{0+}\geq q$. By discussion in the last paragraph of Subsection \ref{subsec:domains}, it is enough to show that $\mathbb{U}(\hat{c})\geq \mathbb{U}(c)$ for all such $c$. Denote $\hat{c}'=(\hat{c}-q)\vee 0\in\mathcal{C}_{\text{inc}}$ and $c'=(c-q)\vee 0\in\mathcal{C}_\text{inc}$. By concavity of $U$, $U(\hat{c}_t)-U(c_t)\geq U'(\hat{c}_t)(\hat{c}_t-c_t)$. Integrating this inequality with respect to $\mathbb{P}\times d\kappa$, we obtain
\begin{equation}\label{r-positive}
\mathbb{U}(\hat{c})-\mathbb{U}(c)\geq \mathbb{E}\left[\int_0^{\hat{T}} U'(\hat{c}_t)(\hat{c}_t-c_t)d\kappa_t\right]=\mathbb{E}\left[\int_0^{\hat{T}} \left(\int_s^{\hat{T}} U'(\hat{c}_t)d\kappa_t\right)(d\hat{c}_s-dc_s)\right].
\end{equation}
Since $U(\hat{c}_t)-U(0)\geq U'(\hat{c}_t)\hat{c}_t\geq 0$ and $\mathbb{E}\left[\int_0^{\hat{T}}U^{-}(t,0)d\kappa_t\right]<\infty$, the process $U'(\hat{c}_t)\hat{c}_t$ is $\mathbb{P}\times d\kappa-$integrable. Hence, we can indeed interchange subtraction and taking expectation in \eqref{r-positive}, and the equality in \eqref{r-positive} is justified.

Since $\hat{c}_t\geq c^y_t$ and $U'$ is decreasing, $\int_s^{\hat{T}} U'(c^y_t)d\kappa_t\geq \int_s^{\hat{T}} U'(\hat{c}_t)d\kappa_t$ for all $s\geq 0$. Moreover, for $s\geq 0$ such that $d\hat{c}_s'>0$, we have $\hat{c}_t=c^y_t$ for all $t>s$ and, since $d\kappa$ has no atoms, $\int_s^{\hat{T}} U'(c^y_t)d\kappa_t= \int_s^{\hat{T}} U'(\hat{c}_t)d\kappa_t$. As a consequence,\begin{equation}\label{eq:chain}
\begin{aligned}
\mathbb{E}\left[\int_0^{\hat{T}} \left(\int_s^{\hat{T}} U'(\hat{c}_t)d\kappa_t\right)(d\hat{c}_s-dc_s)\right]&=\mathbb{E}\left[\int_0^{\hat{T}} \left(\int_s^{\hat{T}} U'(\hat{c}_t)d\kappa_t\right)(d\hat{c}'_s-dc'_s)\right]\\
&\geq\mathbb{E}\left[\int_0^{\hat{T}} \left(\int_s^{\hat{T}} U'(c^y_t)d\kappa_t\right)(d\hat{c}'_s-dc'_s)\right]\\
&=\mathbb{E}\left[\int_0^{\hat{T}} {}^o\left(\int_s^{\hat{T}} U'(c^y_t)d\kappa_t\right)(d\hat{c}'_s-dc'_s)\right].
\end{aligned}
\end{equation}
Finally, combining \eqref{r-positive}, \eqref{eq:chain}, using the defining property \eqref{eq:envelope2} of $c^y$ and the fact that $d\hat{c}'_s>0$ implies $dc^y_s>0$, we obtain
\begin{equation*}
\begin{aligned}
\mathbb{U}(\hat{c})-\mathbb{U}(c)&\geq \mathbb{E}\left[\int_0^{\hat{T}} {}^o\left(\int_s^{\hat{T}} U'(c^y_t)d\kappa_t\right)(d\hat{c}'_s-dc'_s)\right]\geq \mathbb{E}\left[\int_0^{\hat{T}} y^\circ\tilde{Z}_s(d\hat{c}'_s-dc'_s)\right]\\
&=\mathbb{E}\left[\int_0^{\hat{T}} y\tilde{Z}_s(d\hat{c}_s-dc_s)\right]=y\left(\langle\hat{c},Z\rangle-\langle c,Z\rangle\right)\geq 0.
\end{aligned}
\end{equation*}
The uniqueness of optimizer, up to $\mathbb{P}\times d\kappa-$nullsets of $\Omega\times[0,\hat{T})$, follows by strict concavity of $U$ and finiteness of $\mathbb{U}(\hat{c})$.
\end{proof}

\end{appendix}

\textbf{Acknowledgments.} I thank my advisor Mihai S\^irbu for the many helpful discussions on topics related to this paper. The comments of the anonymous referees led to major improvements of the paper and are greatly appreciated. I would also like to thank Bahman Angoshtari for several insightful discussions on ratchet and drawdown constraints.

\textbf{Funding.} The research was supported in part by the National Science Foundation under Grant DMS-1908903. Any opinions, findings, and conclusions or recommendations expressed in this material are those of the authors and do not necessarily reflect the views of the National Science Foundation. Part of this research was performed while the author was visiting the Institute for Mathematical and Statistical Innovation (IMSI), which is supported by the National Science Foundation (Grant DMS-1929348).

%%%%%%%%%%%%%%%%%%%%%%%%%%%%

%\newpage
%\clearpage
%\pagestyle{empty}
\footnotesize
\bibliography{references}
\end{document}